\documentclass[dvips,aos,preprint]{imsart}

\RequirePackage[OT1]{fontenc}
\RequirePackage{amsthm,amsmath}
\RequirePackage[numbers]{natbib}
\RequirePackage[colorlinks,citecolor=blue,urlcolor=blue]{hyperref}

\arxiv{math.PR/0000000}

\usepackage{latexsym,amsmath,amsfonts}
\usepackage{algorithmicx,algpseudocode}
\usepackage{amssymb}
\usepackage{pstricks,pst-node,pst-tree}
\usepackage{thmtools,thm-restate}
\usepackage{graphicx,epsfig,epstopdf}

\usepackage{multirow}

\startlocaldefs
\numberwithin{equation}{section}
\theoremstyle{plain}

\newtheorem{theorem}{Theorem}

\newtheorem{corollary}{Corollary}
\newtheorem{lemma}{Lemma}

\def\R{\mathbb{R}} 

\def\Ep{{\rm E}}
\def\En{{\mathbb{E}_n}}
\def\Gn{{\mathbb{G}_n}}
\def\P{{\rm P}}

\newtheorem{algorithm}{Algorithm}

\renewcommand{\Pr}[1]{\mathbb{P}\left(#1\right)} 

\def\e{\varepsilon}

\def\tr{{\rm Tr}}

\endlocaldefs

\begin{document}

\begin{frontmatter}
\title{Confidence Bands for Coefficients in High Dimensional Linear Models with Error-in-variables}
\runtitle{Inference in Error-in-variables}
\thankstext{T1}{First Version: 4/10/2016; Current Version: \today. We are indebted to Mathieu Rosenbaum and Alexandre B. Tsybakov for several discussions and suggestions on this work. We are also thankful to participants of seminars at Wharton, USC Marshall, CREST and UC Davis for comments.}

\begin{aug}
\author{\fnms{Alexandre} \snm{Belloni}\ead[label=e1]{abn5@duke.edu}},
\author{\fnms{Victor} \snm{Chernozhukov}\ead[label=e2]{vchern@mit.edu}}\\ 
\and
\author{\fnms{Abhishek} \snm{Kaul}\ead[label=e3]{abhishek.kaul@nih.gov}}



\address{The Fuqua School of Business\\
Duke University\\
100 Fuqua Drive\\
Durham, NC 27708\\
\printead{e1}\\
}

\address{Department of Economics\\ Massachusetts Institute of Technology\\
52 Memorial Drive\\
Cambridge, MA 02139\\
\printead{e2}\\
}

\address{Biostatistics \& Computational \\ Biology Branch\\
NIEHS\\
P.O. Box 12233\\
Mail Drop A3-03 \\
Research Triangle Park\\
 NC 27709 \\
\printead{e3}\\
}

\end{aug}

\begin{abstract}
We study high-dimensional linear models with error-in-variables. Such models are motivated by various applications in econometrics, finance and genetics. These models are challenging because of the need to account for measurement errors to avoid non-vanishing biases in addition to handle the high dimensionality of the parameters. A recent growing literature has proposed various estimators that achieve good rates of convergence. Our main contribution complements this literature with the construction of simultaneous confidence regions for the parameters of interest in such high-dimensional linear models with error-in-variables.

These confidence regions are based on the construction of moment conditions that have an additional orthogonal property with respect to nuisance parameters. We provide a construction that requires us to estimate an additional  high-dimensional linear model  with error-in-variables for each component of interest. We use a multiplier bootstrap to compute critical values for simultaneous confidence intervals for a subset $S$ of the components. We show its validity despite of possible model selection mistakes, and allowing for the cardinality of $S$ to be larger than the sample size.

We apply and discuss the implications of our results to two examples and conduct Monte Carlo simulations to illustrate the performance of the proposed procedure.
\end{abstract}


\begin{keyword}
\kwd{honest confidence regions}
\kwd{error-in-variables}
\kwd{high dimensional models}
\end{keyword}

\end{frontmatter}

\section{Introduction}

High-dimensional data sets are now commonplace in a range of fields such as econometrics, finance and genomics. This has motivated the development of a large literature on the estimation of the corresponding parameters of the  models of such data. A key feature of the literature is that these models have a large number of parameters which can be comparable or even exceed the available sample size. Under sparsity assumptions of the high-dimensional parameter vector, penalized methods have been used and proved to be effective in a variety of settings, see, e.g., \cite{fan2011sparse,carvalho2012high,hautsch2014financial}.

In this work we consider high-dimensional linear models with error-in-variables. Such models are challenging because of the need to account for  measurement errors to avoid non-vanishing biases. This has been critical even in the low dimensional setting \cite{Fuller1987} and \cite{CRSC2006}. More recently several authors have considered the high-dimensional linear models with error-in-variables including \cite{LL2009}, \cite{LW}, \cite{SFT2}, \cite{RT1}, \cite{RT2}, \cite{CChen}, \cite{CChen1}, \cite{BRT2014}, \cite{BRT2014b},  \cite{KK2015}, \cite{RZ2015} and \cite{KKCL2016}. These papers propose and analyse different estimators. The main results are rates of convergence in different norms. Under various conditions and suitable choice of penalty parameters, these estimators attain $\ell_q$-rates of convergence of the form
\begin{equation}\label{eq:minimaxbound}\|\hat\beta - \beta_0\|_q \leq C (1+\|\beta_0\|)s^{1/q}\sqrt{\frac{\log p}{n}} \end{equation}
which are minimax optimal, see \cite{BRT2014}. The rate in (\ref{eq:minimaxbound}) highlights the impact of the error-in-variables via the $\ell_2$-norm of $\beta_0$, which is not present in the case where covariates are observed without error, and that consistency can be achieved in high-dimensional settings even if $p\gg n$. However, these estimators are not asymptotically normal and not suitable for the construction of confidence regions with asymptotically correct coverage without imposing stringent assumptions that allow perfect model selection.

Our main contribution is the construction of confidence regions for the parameters of interest in such high-dimensional linear models with error-in-variables. This complements prior work that derived rates of convergence for these models established in the references above. Thus our work is motivated by applications where confidence intervals and/or hypothesis testing is desired instead of prediction accuracy. This is the case in several applications in economics, public health, and genetics. Nonetheless, a direct consequence of the honest confidence intervals is a new estimator that achieves the minimax $\ell_\infty$-rate under weaker design conditions. 

Some definitive theoretical findings on the construction of confidence regions for parameters in high-dimensional models have emerged in recent years. In a high dimensional context \cite{BellChernHans:Gauss,BellChenChernHans:nonGauss} provide uniformly valid inference methods for instrumental variables, high-dimensional linear models have been considered in \cite{BCH2011:InferenceGauss,BelloniChernozhukovHansen2011}, \cite{zhang2014confidence}, \cite{GRD2014}, and \cite{javanmard2014confidence} while non-linear models have been considered in \cite{GRD2014}, \cite{BCW-logistic}, \cite{BCK-LAD}, among others.  
The results in these references are uniformly valid inference over a large set of data generating processes despite of model selection mistakes. Indeed they do not attempt to achieve the oracle property that relies on separation from zero conditions which leads to the lack of uniform validity, see \cite{leeb:potscher:hodges,leeb:potscher:review}. In many of these works the uniform validity of these estimators relies on the use of orthogonal moment functions that can be traced back to Neymann \cite{neyman1965use,Neyman1979} and has been extensively used in various settings, e.g. \cite{bickel:1982}, \cite{cox1987parameter},\cite{newey90}, \cite{newey94}, among others. Such orthogonal moment functions reduce the impact of the estimation of nuisance parameters on the estimation of the parameters of interest. In particular, under suitable conditions this allows for $\sqrt{n}$-consistent estimates that are asymptotically normal despite the use of non-regular estimators for the nuisance parameters that unavoidably arise due to the high-dimensionality.

Here we build upon recent results of estimators for high-dimensional linear models with error-in-variables to construct orthogonal moment functions that will allow us to construct (simultaneous) confidence intervals for these parameters. It follows that the error-in-variables also impacts the construction of the orthogonal moment functions (which can be seen as a de-biasing step) and need to be accounted for. We establish a linear representation for the estimation error. This allow us to show the $\sqrt{n}$-consistency and asymptotic normality for each individual estimate which directly leads to the construction of confidence intervals that are uniformly valid over a large class of data generating processes. Moreover, simultaneous confidence intervals can be constructed based on critical values from a multiplier bootstrap procedure which validity is derived building upon recent results on high-dimensional central limit theorems and bootstrap theorems established in \cite{chernozhukov2013gaussian}, \cite{chernozhukov2014clt}, \cite{chernozhukov2012gaussian}, \cite{chernozhukov2012comparison}, and \cite{chernozhukov2015noncenteredprocesses}. We establish its validity under conditions that allows for simultaneous confidence intervals over a larger number of components than the available sample size.

We also fully characterize the impact of using an estimate $\hat \Gamma$ of the variance of the error-in-variables $\Gamma$ which is important in many applications. Although we can see $\hat \Gamma$ as additional nuisance parameter, one cannot achieve the orthogonal property with respect to $\hat \Gamma$. It turns out that the impact of using an estimate is non-negligible. We further show how to adjust the multiplier bootstrap when the estimator $\hat \Gamma$ itself admits a linear representation which again leads to uniformly valid confidence regions. This approach seems to be new and of independent interest in other high-dimensional problems.

We apply and discuss the implications of our results to two examples. We provide simple sufficient conditions for the validity of the results. The first application is the estimation of the inverse covariance matrices in high-dimensions with error-in-variables. Such problem, without error-in-variables, has been motivated in a variety of applications including social network analysis, climate data analysis, and finance. Recent work that provides estimators with rates of convergence include \cite{meinshausen2009lasso,yuan2010high}. Recently, based on de-biasing ideas, \cite{jankova2015confidence} proposed a methodology for statistical inference for low-dimensional parameters of sparse precision matrices in a high-dimensional settings when there is no error-in-measurements. In the case with error-in-variables, \cite{LW} provides an estimator with rates of convergence. The second application consists of missing data at random which has motivated a lot of the literature in error-in-variables models even in the low dimensional case. In the high-dimensional case, estimators for this case with good $\ell_q$-rates of convergence have been proposed in \cite{RT2}, \cite{LW}, \cite{BRT2014} and \cite{KKCL2016}. As shown in Section \ref{sec:examples}, this is an application where the estimator of the variance of the error-in-variables admits a linear representation. This allows us to adjust the bootstrap procedure and construct regions that account for using the estimate $\hat\Gamma$. Our results complement such findings by providing new estimates and associated confidence bands for potentially a high-dimensional vector of parameters.

The rest of this paper is organized as follows. Section \ref{sec:Model} describes the model under consideration and describes the proposed methodology to construct (simultaneous) confidence regions for the parameters of high-dimensional linear models with error-in-variables. In Section \ref{sec:Main} we provide assumptions and our main theoretical results including the uniform validity of the confidence regions. We present examples that illustrate our results in Section \ref{sec:examples} and simulations in Section \ref{sec:simulations}. All proofs are relegated to the appendix.

\section{Model and Method for Confidence Regions}\label{sec:Model}

We consider i.i.d. observations from a regression model with observation error in the design:
\begin{equation}\label{model} y_i=x_i^T\beta_0+\xi_i, \ \ \ z_i = x_i + w_i, \ \ \ i=1,\ldots, n\end{equation}
where we observe the response variable $y_i$ and the $p$-dimensional vector $z_i$, but we do not observe the covariates $x_i$. The vector $w_i$ and the scalar $\xi_i$ are unobserved zero-mean random vectors. The vector $\beta_0 \in \mathbb{R}^p$ is a vector of unknown parameters to be estimated where the dimension $p$ can be much larger than the sample size $n$, and $\beta_0$ is sparse with $s$ non-zero components, i.e. $\|\beta_0\|_0= s$. The measurement error $w$ satisfies $\Ep[x_jw_k]=0$ and its covariance matrix, $\Gamma = {\rm cov}(w) = {\rm diag}(\Ep[w_{1}^2],\ldots,\Ep[w_{p}^2])$, is known. \footnote{The case of unknown $\Gamma$ is discussed in Section \ref{sec:estimatedGamma}.}

In what follows we define a pseudo-likelihood function $$\ell(y,z,\beta)= -\frac{1}{2}\beta^T(zz^T-\Gamma)\beta+yz^T\beta.$$ Given the zero-mean conditions associated with model (\ref{model}), direct calculations show that the vector of parameters $\beta_0$ solves the following moment condition
\begin{equation}\label{eq:identification} \Ep[ \partial_{\beta} \ell(y,z,\beta)] = \Ep[ z(y-z^T\beta)+\Gamma\beta]=0 \end{equation}
where the term $\Gamma\beta_0$ corrects the bias that arises from using the (noisy) covariates $z$ instead of the unobserved $x$. 

Next we propose an estimator which is asymptotically normal and a bootstrap method to compute confidence regions for the parameters $\beta_0$. To achieve that we will use a score function $\psi_j$ tailored for each $\beta_{0j}$ with the form
$$ \psi_j(y,z,\beta_j,\eta^j)= \partial_{\beta_j} \ell(y,z,\beta) - (\mu^j)^T \partial_{\beta_{-j}}\ell(y,z,\beta) = (e^j-\mu^j)^T\partial_{\beta}\ell(y,z,\beta) $$
where $e^j$ is the jth canonical basis,\footnote{$e^j_j=1$ and $e^j_k=0$ if $k\neq j$.}  $\mu^j$ is a $p$-dimensional vector with a zero in the jth component, and $\eta^j=(\beta_{-j},\mu^j)$ collects all the nuisance parameters for $\psi_j$. Note that for any choice of $\eta^j=(\beta_{0,-j},\mu^j)$ we have by (\ref{eq:identification}) that
$$ \Ep[\psi_j(y,z,\beta_{0j},\eta^j)] = 0.$$
We will choose $\eta^j_0=(\beta_{0,-j},\mu^j_0)$ so that the function $\psi_j$ also satisfies the following  orthogonality condition
 \begin{equation}\label{ortho:prop} \partial_{\eta^j}\left.\Ep[\psi_j(y,z,\beta_{0j},\eta^j)]\right|_{\eta^j=\eta^j_0}=0. \end{equation}
Condition (\ref{ortho:prop}) makes the procedure first-order insensitive to the estimation error of the nuisance parameters $\eta^j_0$. Importantly, we will construct one such score function for each component $j\in S \subseteq \{1,\ldots,p\}$ we would like to estimate. We will show that this allows for enough adaptivity that will lead to regular and asymptotically normal estimators despite of model selection mistakes and the presence of high-dimensional nuisance parameters. \footnote{It is straightforward to verify that (\ref{eq:identification}) does not satisfy the orthogonality condition (\ref{ortho:prop}) unless covariates were orthogonal to each other.}

Letting $J=\partial_\beta \left.\Ep[\partial_{\beta^T}\ell(y,z,\beta)]\right|_{\beta=\beta_0}=\Ep[zz^T]-\Gamma=\Ep[xx^T]$, it follows that the desired orthogonality property (\ref{ortho:prop}) is achieved if
$\mu_0^j$ solves the system of equations $$J_{j,-j}-\mu J_{-j,-j} = 0.$$ 

In the estimation of $\beta_{0j}$, we will preliminary (possibly non-regular) estimators $\hat \eta^j = (\hat \beta_{-j}; \hat \mu^j)$ of the nuisance parameters in $\eta^j_0 = (\beta_{0,-j};\mu_0^j)$. Thus, using $\hat\eta^j$, the score function $\psi_j$ used to estimate $\beta_{0j}$ is given by
\begin{equation}\label{ortho:j}
\begin{array}{rl}
\psi_j(y_i,z_i,\theta,\hat\eta^j)  & = z_{ij}(y_i-z_{ij}\theta-z_{i}^T\hat \beta_{-j})+\Gamma_{jj}\theta\\
& -(\hat \mu^j)^T\{z_{i,-j}(y_i-z_{ij}\theta-z_{i,-j}^T\hat \beta_{-j})+\Gamma_{-j,-j}\hat \beta_{-j} \} \\
& = (e^j-\hat\mu^j)^T\{z_{i}(y_i-z_{ij}\theta-z_{i,-j}^T\hat \beta_{-j})+\Gamma(\theta e^j+\hat\beta_{-j})\}\end{array}\end{equation}

In order to estimate the nuisance parameter $\eta^j_0=(\beta_{0,-j},\mu_0^j)$, by definition $\beta_0$ can be estimated via the methods recently proposed in the literature for high-dimensional linear models with error-in-measurements. Moreover, the vector of parameters $\mu_0^j$ is such that
\begin{equation}\label{aux:model} x_{ij} =  x_{i,-j}^T\mu_0^j + \nu^j_i, \ \ \ z_i = x_i + w_i, \ \ \Ep[x_{i,-j}\nu^j_i]=0, \ \ i=1,\ldots,n \end{equation} where only $(z_i)_{i=1}^n$ are observed. Thus the nuisance parameter $\mu^j_0$ is also characterized by a high-dimensional linear model with error-in-variables similar to (\ref{model}). \footnote{Indeed the corresponding model (\ref{model}) would have the response variable $\tilde y=z_j$ and noise $\tilde \xi=w_{j}+\nu^j$. Here we exploit that $\Gamma$ is diagonal, so that the same moment condition still works since $\Gamma_{j,-j}=0$, $$\begin{array}{rl}
\Ep[z_{-j}\{z_j-(\mu_0^j)^Tz_{-j}\}] & =\Ep[\{x_{-j}+w_{-j}\}\{w_j+\nu^j-(\mu_0^j)^Tw_{-j}\}]=\Gamma_{j,-j}-(\mu_0^j)^T\Gamma_{-j,-j}.\end{array}$$
}
Therefore, the estimation of the nuisance parameters requires the estimation of high-dimensional linear regression models with errors in measurements. Under various conditions, different estimators have been proposed in the literature and shown to have good rates of convergence in the $\ell_1$ and $\ell_2$-norms, see \cite{CChen,CChen1,LW,RT1,RT2,BRT2014}. Our results will apply to many of these estimators that are computed via regularization (typically $\ell_1$-penalty). 

Algorithm \ref{algo:1} summarizes the proposed estimator. We will provide conditions under which these rates of convergence suffice to establish asymptotic normality and $\sqrt{n}$-consistency of $\check\beta_j$ when combined with the orthogonality (\ref{ortho:prop}).

\begin{algorithm}[Estimation based on Orthogonal Score Functions]\label{algo:1} ~\\
\noindent  \emph{Step 1}. Compute an estimator $\hat\beta$ of $\beta_0$ in (\ref{model}) via regularization.\\ 
For each $j \in S \subset \{1,\ldots,p\}$: \\
\enspace \emph{Step 2}. Compute an estimator $\hat\mu^j$ of $\mu_0^j$ in (\ref{aux:model}) via regularization.\\ 
\enspace \emph{Step 3}. Construct $\psi_j$ as defined in (\ref{ortho:j}) with $\hat\eta^j=(\hat\beta_{-j};\hat\mu^j)$. Compute $\check \beta_j$ as \\
$$ \check\beta_j \in \arg\min_{\theta \in \mathbb{R}} \left|\frac{1}{n}\sum_{i=1}^n\psi_j(y_i,z_i,\theta,\hat\eta^j)\right|$$
\end{algorithm}

In this setting, the minimization in Step 3 has a closed form solution given by
$$ \check\beta_j := \frac{\hat \Sigma_j^{-1}}{n}\sum_{i=1}^n(z_{ij}-z_{i,-j}^T\hat\mu^j)(y_i-z_{i,-j}^T\hat\beta_{-j})-(\hat \mu^j)^T\Gamma_{-j,-j}\hat\beta_{-j}$$
where $\hat \Sigma_j := \left\{\frac{1}{n}\sum_{i=1}^n (z_{ij}-z_{i,-j}^T\hat\mu^j)z_{ij} - \Gamma_{jj} \right\}$.

Due to the orthogonality condition (\ref{ortho:prop}), under mild conditions confidence intervals can be constructed based on the normal approximation, namely
\begin{equation}\label{region:one} \sqrt{n}\sigma_j^{-1}(\check\beta_j-\beta_{0j})\rightsquigarrow N(0,1)\end{equation}
where $\sigma_j^2 = \Sigma_j^{-2}\Ep[\psi_j^2(y,z,\beta_{0j},\eta^j_0)]$, for $\Sigma_j:=\left.\partial_{\beta_j}\Ep[\psi_j(y,z,\beta_{j},\eta^j_0)]\right|_{\beta_j=\beta_{0j}}=\Ep[(z_j-z_{-j}^T\mu^j_0)z_j-\Gamma_{jj}]$. The quantity $\hat\Sigma_j$ is an estimator for $\Sigma_j$ and we estimate $\sigma_j$ by a plug-in rule \footnote{An alternative estimator would use $\check\beta_j$ instead of $\hat\beta_j$.}
\begin{equation}\label{est:sigmaj}
\hat\sigma_j^2 = \hat \Sigma_j^{-2}\frac{1}{n}\sum_{i=1}^n\psi_j^2(y_i,z_i,\hat\beta_j,\hat\eta^j).\end{equation}
Next we construct simultaneous confidence bands for a subset $S\subseteq \{1,\ldots,p\}$,  which cardinality $|S|\geq 2$ is potentially larger than $n$. That is, for a given $\alpha\in(0,1)$, we choose a critical value $c_{\alpha,S}^*$ such that with probability converging to $1-\alpha$ we have
\begin{equation}\label{region:simultaneous} \check \beta_j - c_{\alpha,S}^* \frac{\hat\sigma_j}{\sqrt{n}}\leq \beta_{0j} \leq \check \beta_j + c_{\alpha,S}^* \frac{\hat\sigma_j}{\sqrt{n}} \ \ \mbox{for all} \ j\in S. \end{equation}

Critical values for simultaneous confidence regions can be constructed based on the multiplier bootstrap following the approach in \cite{chernozhukov2013gaussian,chernozhukov2014clt,chernozhukov2012gaussian,chernozhukov2012comparison,chernozhukov2015noncenteredprocesses}. Letting $\hat \psi_j(y_i,z_i):=-\hat\sigma_j^{-1}\hat \Sigma_j^{-1}\psi(y_i,z_i,\check\beta_j,\hat\mu^j)$, define the vector $\widehat{\mathcal{G}}$ as
\begin{equation}\label{def:mb} \widehat{\mathcal{G}}_j:=\frac{1}{\sqrt{n}}\sum_{i=1}^n g_i \hat \psi_j(y_i,z_i), \ \ j=1,\ldots,p, \end{equation}
where $(g_i)_{i=1}^n$ are independent standard normal random variables independent from $(y_i,z_i)_{i=1}^n$. We compute the critical value $c_{\alpha,S}^*$ for a subset $S\subseteq \{1,\ldots,p\}$  as the $(1-\alpha)$-quantile of the conditional distribution of $$\max_{j\in S} |\widehat{\mathcal{G}}_j|$$ given the data $(y_i,z_i)_{i=1}^n$.
Theoretical results will allow the cardinality of $S$ to grow with the sample size potentially exceeding it but requiring $\log^7 |S| = o(n)$ among other technical conditions.

\section{Main Results}\label{sec:Main}

In this section we state our assumptions and main theoretical results for the validity of the confidence regions based on (\ref{region:one}) and (\ref{region:simultaneous}). The first assumption regarding the data generating process is as follows.\footnote{Recall that for $\gamma>0$, the random variable $\eta$ is said to be
 {\it $\gamma$-subgaussian} if, for all $t\in\mathbb{R}$,
$\Ep[\text{exp}(t\eta)]\leq \text{exp}(\gamma^2t^2/2).$ Similarly, a random vector $\zeta\in \R^p$ is said to be
 {\it $\gamma$-subgaussian} if the inner products $(\zeta, v)$ are  $\gamma$-subgaussian for any $v\in \R^p$ with $\|v\|=1$.}

~\\
{\bf Condition A.} {\it (i) The vectors $\{(y_i,z_i,x_i,w_i,\xi_i), i=1,\ldots,n\}$ are i.i.d. observations obeying (\ref{model}). (ii) The vector $x_i$ is drawn from a subgaussian distribution with parameter $\sigma_x^2 \leq C$ and its covariance matrix $\Omega$ has eigenvalues between positive constants $c$ and $C$ independent of $n$. (iii) The elements of the random noise vector $\xi$ are independent zero-mean subgaussian random variables with variance parameter $\sigma^2_\xi\leq C$. (iv) The noise vector $w_i$ is a zero-mean subgaussian random vector with variance parameter $\sigma_w^2\leq C$ satisfying $\Ep[x_jw_k]=0$ for all $1\leq j<k\leq p$. (v) The covariance matrix of $w$ is diagonal and known, i.e. $\Gamma = {\rm cov}(w) = {\rm diag}(\Ep[w_{1}^2],\ldots,\Ep[w_{p}^2])$. }
~\\

Conditions A(i)-(iii) are standard in high-dimensional linear regression models. In particular they guarantee that the (unobserved) design matrix $\frac{1}{n}\sum_{i=1}^nx_ix_i^T$ has well behaved $\ell_q$-sensitivity and restricted eigenvalues. These quantities are useful to establish convergence in the $\ell_q$-norms for penalized estimators as shown in \cite{gautier2013pivotal,gautier2011high} and it has been used in \cite{BRT2014} for the error in measurements model (\ref{model}). It is related and weakens conditions associated with the restricted eigenvalue condition \cite{BickelRitovTsybakov2009}. Condition A(iv)-(v) provides a way to identify the unknown parameter $\beta_0$ despite of the error in measurements. In the literature several different conditions are used. The case of unknown $\Gamma$ is considered in Section \ref{sec:estimatedGamma} where an estimator $\hat \Gamma$ is available. 

In particular, under Condition A and suitable choice of parameters, different estimators  $\hat\beta$ in the literature achieve optimal rates of convergence in $\ell_1$ and $\ell_2$-norms with probability approaching 1 as the sample size increases. (The same holds for estimators $\hat\mu^j$ of $\mu_0^j$.) We will be agnostic about the choice of such estimators (which might be different as some assumptions change) and require the following condition. In what follows $\Delta_n\to 0$ is a fixed sequence and $C$ is a fixed constant.

~\\
{\bf Condition B.} {\it Let $S \subseteq \{1,\ldots,p\}$, $|S|\geq 2$, and $s=s_n\geq 1$. We have $\|\beta_0\|_0\leq s$ and $\max_{j\in S}\|\mu^j_0\|_0\leq s$. With probability $1-\Delta_n$ we have that the estimators $\hat\beta$ and $(\hat\mu^j)_{j\in S}$ satisfy for $q \in \{1,2\}$
\begin{itemize}
\item[(i)]$\displaystyle \|\hat\beta-\beta_0\|_q \leq C (1+\|\beta_0\|)s^{1/q}\sqrt{\frac{\log p}{n}}, \ \ \mbox{and} \ \ \|\hat\beta\|_0 \leq Cs; $
\item[(ii)]$\displaystyle   \|\hat\mu^j-\mu_0^j\|_q \leq C (1+\|\mu_0^j\|)s^{1/q}\sqrt{\frac{\log p}{n}},\ \ \mbox{and} \ \ \|\hat\mu^j\|_0 \leq Cs \ \ \mbox{for all} \ j\in S$\end{itemize} }
~\\

Condition B assumes that the vectors $\beta_0$ and $\mu_0^j$ are sparse. The $\ell_2$-rate combined with the sparsity bound  immediately imply an $\ell_1$-rate of convergence. We note that several estimators were shown to satisfy the required $\ell_1$ and $\ell_2$-rate of convergence. However, sparsity guarantees have not been common in the literature. Nonetheless, these rates of convergence and the sparsity condition in $\beta_0$ can be used to show that hard thresholding these estimators yields the desired sparsity requirements and preserves the rates of convergence, see \cite{BCKRT2016a} for a detailed analysis. We note that this will cover the estimators proposed in \cite{LW,BRT2014} as well as post-selection refitted versions of these estimators.

Next we state conditions on the growth of various parameters that characterize the model. In what follows $\delta_n\to 0 $ is a fixed sequence.

~\\
{\bf Condition C.}{\it We have that $\|\beta_0\|_\infty \leq C$ and $$(1+\|\beta_0\|)(1+\max_{j\in S}\|\mu_0^j\|)s \log (np) \leq \delta_n\sqrt{n}.$$}
~\\
The growth restriction of $s$, $p$ and $n$ are compatible with the requirements to construct confidence intervals for high-dimensional linear regression models without errors-in-variables. We note that a consequence of Condition A(ii) is that $\max_{j\leq p}\|\mu_0^j\|\leq C$ which could be used to simplify Condition C above. The current statement of Condition C highlights how the norm of $\mu_0^j$ would impact the requirements if Condition A(ii) is relaxed.

The result below is one of our main results. It establishes a linear representation for the estimators despite of the high-dimensional nuisance parameters and model selection mistakes. We note that the sequence $\delta_n\to 0$ defined in Condition C controls the approximation error of the linear representation.

\begin{theorem}[Uniform Linear Representation]\label{thm:main}
Under Conditions A, B and C, uniformly over $j\in S$ we have
$$\sqrt{n}\sigma_j^{-1}(\check\beta_j-\beta_{0j}) = \frac{1}{\sqrt{n}}\sum_{i=1}^n \bar\psi_j(y_i,z_i) + O_\P(\sigma_j^{-1}\Sigma_j^{-1}\delta_n)$$
where $\bar\psi_j(y,z)=-\sigma_j^{-1}\Sigma_j^{-1}\psi_j(y,z,\beta_{0j},\eta_0^j)$ satisfies $\Ep[\bar\psi_j^2(y,z)]=1$.
\end{theorem}

Theorem \ref{thm:main} holds uniformly over the class of data generating processes that satisfy  Conditions A, B and C. In particular, it allows for possible model selection mistakes that are likely to happen when coefficients are near zero.
Theorem \ref{thm:main} can be used to establish useful estimation results. In particular, the following $\ell_\infty$-rate of convergence without additional assumptions on the design matrix.

\begin{corollary}[$\ell_\infty$-rate of Convergence]\label{thm:Linfty}
Under Conditions A, B and C with $S=\{1,\ldots,p\}$, and $(1+\|\beta_0\|)\log(np)\log^{1/2} n \leq \delta_n \sqrt{n}$, Algorithm \ref{algo:1} yields an estimator $\check\beta$ such that with probability $1-\varepsilon-o(1)$
$$\|\check\beta - \beta_0\|_\infty \leq C(1+\|\beta_0\|) \sqrt{\frac{\log (Cp/\varepsilon)}{n}}$$
for some constant $C$ independent of $n$.\end{corollary}

Corollary \ref{thm:Linfty} establishes a $\ell_\infty$-norm rate of convergence for $\check\beta$ that matches the associated minimax lower bound for the $\ell_\infty$-rate of convergence established in Theorem 3 of \cite{BRT2014}. We note that this is achieve under weaker design conditions than in \cite{BRT2014} which required $\kappa_\infty(s,3) \geq c$ (implied by vanishing mutual coherence, i.e. near zero correlation across unobserved covariates $x$).

Another consequence of Theorem \ref{thm:main} is the construction of confidence intervals for each component as $\sqrt{n}\sigma_j^{-1}(\check\beta_j-\beta_{0j})\rightsquigarrow N(0,1)$. Importantly, it holds uniformly over data generating processes with arbitrary small coefficients. Indeed, the orthogonality condition mitigates the impact of model selection mistakes which are unavoidable for those components.

\begin{corollary}[Componentwise Confidence Intervals]\label{cor:PointwiseInference}
Let $\mathcal{M}_n$ be the set of data generating processes that satisfies Conditions A, B and C for a fixed $n$. We have that
$$ \lim_{n\to \infty} \sup_{\mathcal{M}\in \mathcal{M}_n} \max_{j\leq p} \left|\P_\mathcal{M} \left( \sqrt{n}|\check \beta_j - \beta_{0j}| \leq \Phi^{-1}(1-\alpha/2)\sigma_j \right)  - (1-\alpha)\right| = 0$$
Further, the result also holds when $\sigma_j$ is replaced by $\hat\sigma_j$ as defined in (\ref{est:sigmaj}).
\end{corollary}

Next we turn to simultaneous confidence bands over $S\subseteq\{1,\ldots,p\}$ components of $\beta_0$. We allow for the cardinality of $S$ to also grow with the sample size (and potentially $S=\{1,\dots,p\}$). We use central limit theorems for high-dimensional vectors, see \cite{chernozhukov2015noncenteredprocesses} and the references therein. The following result provides sufficient conditions under which the multiplier bootstrap procedure based  on (\ref{def:mb}) yields (honest) simultaneous confidence bands that are asymptotically valid

\begin{theorem}\label{thm:inference}
Let $\mathcal{M}_n$ be the set of data generating processes that satisfies Conditions A, B and C for a fixed $n$. Furthermore, suppose that:\\
(i) $\delta_n\log^{3/2}(|S|)(1+\|\beta_0\|)(1+\max_{j\in S}\|\mu^j_0\|)=o(1)$, and\\
(ii) $\max_{j\in S}\{\sigma_j^{-1}\Sigma_j^{-1}(1+\|\beta_0\|)(1+\|\mu^j_0\|)\}^4\log^7 |S| =o(n)$. \\
For the critical value $c_{\alpha,S}$ computed via the multiplier bootstrap procedure, we have that
$$ \lim_{n\to \infty} \sup_{\mathcal{M}\in \mathcal{M}_n} \left|\P_{\mathcal{M}}{\small \left( \check \beta_j -  \frac{c_{\alpha,S}^*\hat\sigma_j}{\sqrt{n}}\leq \beta_{0j} \leq \check \beta_j + \frac{c_{\alpha,S}^*\hat \sigma_j}{\sqrt{n}},  \forall j\in S\right)}  - (1-\alpha)\right| = 0$$
\end{theorem}

Theorem \ref{thm:inference} establishes the asymptotic validity of the confidence regions. The results are uniformly valid across models that satisfy the stated conditions. In particular, we allow for (sequence of) models where model selection mistakes are unavoidable.

\subsection{Estimated Covariance of Error-in-measurement}\label{sec:estimatedGamma}

In this section we discuss the case in which the covariance matrix $\Gamma$ is diagonal but unknown. In this case we follow the literature that assumes the availability of an estimator $\hat\Gamma$. The following condition summarizes the properties of such estimator. Recall that we denote by $\Delta_n$ a (fixed) sequence going to zero.

~\\

{\bf Condition D.} {\it With probability $1-\Delta_n$ the estimator $\hat\Gamma$ is a diagonal matrix and satisfies
$\|\hat\Gamma-\Gamma\|_\infty \leq C\sqrt{\log(np)/n}$. }

~\\

Condition D is standard and it is satisfied in a variety of applications, e.g. \cite{RT1}. This condition will suffice for us to derive a new linear representation that accounts for the use of an estimate of $\Gamma$. Thus the proposed algorithm uses $\hat\Gamma$ instead of $\Gamma$ in Algorithm \ref{algo:1}.

\begin{algorithm}[Estimation based on estimated $\Gamma$]\label{algo:2} ~\\
\noindent  \emph{Step 1}. Compute an estimator $\hat\beta$ of $\beta_0$ in (\ref{model}) via regularization.\\
For each $j \in S \subset \{1,\ldots,p\}$: \\
\enspace \emph{Step 2}. Compute an estimator $\hat\mu^j$ of $\mu_0^j$ in (\ref{aux:model}) via regularization.\\
\enspace \emph{Step 3}. For $\hat \Sigma_j :=  \frac{1}{n}\sum_{i=1}^n (z_{ij}-z_{i,-j}^T\hat\mu^j)z_{ij} - \hat\Gamma_{jj} $, compute $\check \beta_j$ as \\
$$ \check\beta_j := \frac{\hat \Sigma_j^{-1}}{n}\sum_{i=1}^n(z_{ij}-z_{i,-j}^T\hat\mu^j)(y_i-z_{i,-j}^T\hat\beta_{-j})-(\hat \mu^j)^T\hat \Gamma_{-j,-j}\hat\beta_{-j}$$

\end{algorithm}

Next we state our main results of this section.

\begin{theorem}\label{thm:mainEst}
Suppose that Conditions A, B,  C, and D hold. Then, uniformly over $j\in S$, the estimator based on Algorithm \ref{algo:2} satisfies
$$\begin{array}{rl}
\displaystyle \sqrt{n}\sigma_j^{-1}(\check\beta_j-\beta_{0j}) & \displaystyle =  \frac{1}{\sqrt{n}}\sum_{i=1}^n \bar\psi_j(y_i,z_i)+ O_\P(\sigma_j^{-1}\Sigma_j^{-1}\delta_n)\\
\\
&\displaystyle \ \ \ \   +  \sigma_j^{-1}\Sigma_j^{-1}(e^j-\mu^j_0)^T\sqrt{n}(\hat\Gamma - \Gamma)\beta_0 \end{array}$$
where $\bar\psi_j(y_i,z_i)=-\sigma_j^{-1}\Sigma_j^{-1}\psi_j(y_i,z_i,\beta_0,\mu_0^j)$.
\end{theorem}

Theorem \ref{thm:mainEst} explicitly characterize the impact of using an estimate $\hat\Gamma$ of the covariance matrix $\Gamma$. This result also highlights the fact we do not have an orthogonality condition for the (nuisance) parameter $\Gamma$ as we have for $\mu^j_0$ and $\beta_{0,-j}$. In principle, the estimation error in $\hat\Gamma$ could dominate and dictate the rate of convergence. Next we state a regularity condition on $\hat \Gamma$ that is satisfied in some applications.

~\\

{\bf Condition E.} {\it $\hat\Gamma$ admits a linear representation, namely uniformly over $j\leq p$ we have
$$ \sqrt{n}(\hat \Gamma_{jj} - \Gamma_{jj}) = \frac{1}{\sqrt{n}}\sum_{i=1}^n\varphi_j(z_i) + O_\P(\delta_n)$$
for the sequence $\delta_n \to 0$ defined in Condition C.}

~\\

For applications in which Condition E also holds (see Section \ref{Example:MAR} below), we can rewrite the estimation error as a sum of zero mean terms (up to a negligible term). The following corollary of Theorem \ref{thm:mainEst} summarizes this observation.

\begin{corollary}\label{cor:estmain}
Suppose that Conditions A, B,  C,  D and E hold. Then, uniformly over $j\in S$, the estimator based on Algorithm \ref{algo:2} satisfies
\begin{equation}\label{eq:LRerror}
\begin{array}{rl}
\displaystyle \sqrt{n}\sigma_j^{-1}(\check\beta_j-\beta_{0j})  = \frac{1}{\sqrt{n}}\sum_{i=1}^n\{ \bar\psi_j(y_i,z_i)-\sigma_j^{-1}\Sigma_j^{-1}(e^j-\mu^j_0)^T\varphi(z_i)\beta_0\} \\
\\
\displaystyle + O_\P(\sigma_j^{-1}\Sigma_j^{-1}\delta_n)\end{array}\end{equation}
\end{corollary}

The representation (\ref{eq:LRerror}) can be used to prove the validity of a multiplier bootstrap procedure to construct valid critical values for simultaneous confidence intervals when only an estimate $\hat \Gamma$ is available.

\section{Examples}\label{sec:examples}

In this section we apply our results to specific context that generates error-in-measurements (\ref{model}).

\subsection{Graphical Model Inverse Covariance Estimation}\label{Example:GM}

In this example we consider the estimation of a Gaussian graphical model in a high-dimensional setting. \cite{meinshausen2009lasso,yuan2010high,jankova2015confidence}. It is well known that the conditional independence structure is determined by the precision matrix (inverse of the covariance matrix). In particular, conditional independence between components $(j,k)$ is characterized by a zero in the corresponding entry of the precision matrix. Another convenient way to characterize conditional independence between $(j,k)$ components is through a zero in the $k$th entry of the  vector $\theta^j$ in the linear model
$$ x_j=x_{-j}^T\theta^j_0+\varepsilon^j$$
where $\varepsilon^j$ is a vector of i.i.d. Gaussian random variables $\Ep[\varepsilon^jx^{-j}]=0$ for each $j=1,\ldots,p$. In this case, estimators with good $\ell_2$-rates of convergence for the columns of the inverse of the covariance matrix have been obtained in the literature, \cite{meinshausen2009lasso,yuan2010high}. Recently, \cite{jankova2015confidence} obtained confidence intervals for the precision matrix using de-biasing ideas related to \cite{GRD2014} using the graphical Lasso.

In the case with error-in-variables, Section 3.3 of \cite{LW} explicitly works out this case when $w$ is Gaussian showing how it can be embedded within model (\ref{model}). Furthermore, under very mild sparsity conditions, they derived rates of convergence for an estimator that combined all the estimates and projects (via $\ell_1$-minimization) into the space of symmetric matrices.

The tools we developed here allow us to consider the case that we do not fully observe the $x$ variables but instead $z=x+w$ where $w$ is subgaussian but possibly non-Gaussian. In particular, with known $\Gamma$, based on Theorem \ref{thm:main} and \ref{thm:inference}, we can directly construct simultaneous confidence intervals for all coefficients of $(\theta^j_0,j=1,\ldots,p)$. That is useful in identifying pairs that are candidates for being conditional independent (whose confidence intervals contain zero) and pairs which we are at least $1-\alpha$ confident that are not conditionally independent.

\subsection{Missing Values at Random}\label{Example:MAR}

In this example we are interested on the model (\ref{model}) where the additive error-in-measurements represents missing data and $p$ increasing. We follow the framework discussed in \cite{RT2}.

We observe $(y_i,\tilde z_{i}, \eta_i, i=1,\ldots,n)$ where
$$ \tilde z_{ij} = x_{ij}\eta_{ij}, \ \ \mbox{with} \ \eta_{ij} \ \mbox{i.i.d. Bernoulli with parameter} \ 1-\pi.$$ When $\eta_{ij}=0$ it indicates that we are missing the observation $x_{ij}$. We cast this into our setting by writing
$$z_{ij} = x_{ij}+ w_{ij}, \ \ \mbox{where} \ z_{ij}=\frac{{\tilde z}_{ij}}{1-\pi} \ \ \mbox{and} \ \ w_{ij}=x_{ij}\frac{\eta_{ij}-(1-\pi)}{1-\pi}.$$
In this case we have $\Gamma_{jj} = \Ep[ w_{ij}^2] = \Ep[ x_{ij}^2] \frac{\pi}{1-\pi} = \Ep[\tilde z_{j}^2] \frac{\pi}{(1-\pi)^2}$ and the estimator $\hat\Gamma_{jj}$ is given by
$$ \hat\Gamma_{jj} =\frac{1}{n}\sum_{i=1}^n \tilde z_{ij}^2 \frac{\hat\pi}{(1-\hat\pi)^2}$$
where $\hat \pi= \frac{1}{np}\sum_{i=1}^n \sum_{j=1}^p1\{\eta_{ij}=0\}$ is a consistent estimate of $\pi$. In fact since $\hat \pi = \pi + O_\P( (np)^{-1/2} )$. It follows that uniformly over $j=1,\ldots,p$
$$\begin{array}{rl}
\sqrt{n}(\hat\Gamma_{jj}-\Gamma_{jj}) & = \frac{1}{\sqrt{n}}\sum_{i=1}^n \left\{\tilde z_{ij}^2 \frac{\pi}{(1-\pi)^2} - \Ep[\tilde z_{ij}^2]\frac{\pi}{(1-\pi^2)}\right\} \\
& + \frac{1}{\sqrt{n}}\sum_{i=1}^n \tilde z_{ij}^2 \left\{ \frac{\hat\pi}{(1-\hat\pi)^2}- \frac{\pi}{(1-\pi)^2} \right\} \\
& = \frac{1}{\sqrt{n}}\sum_{i=1}^n \varphi_j(\tilde z_{ij}) + O_\P\left(p^{-1/2}\max_{j\leq p}\frac{1}{n}\sum_{i=1}^n \tilde z_{ij}^2\right).\\
\end{array}
$$
It follows that under mild conditions $\max_{j\leq p}\frac{1}{n}\sum_{i=1}^n \tilde z_{ij}^2\leq C$ with probability $1-o(1)$.

Furthermore note that using the estimate $$\hat \varphi_j(\tilde z_i) = \{\tilde z_{ij}^2 - (\mbox{$\frac{1}{n}\sum_{i=1}^n\tilde z_{ij}^2)$}\}\frac{\hat\pi}{(1-\hat\pi)^2}$$ of the score function $\varphi_j(\tilde z_i)$ has negligible impact in the multiplier bootstrap procedure. Indeed, we have
$$\begin{array}{ll}
 \displaystyle\max_{j\leq p} \left|\frac{1}{\sqrt{n}}\sum_{i=1}^n g_i \{ \hat \varphi_j(\tilde z_i) - \varphi_j(\tilde z_i) \} \right|& \\
  \displaystyle =  \max_{j\leq p}\left| \frac{1}{\sqrt{n}}\sum_{i=1}^n g_i \{ \mbox{$(\frac{1}{n}\sum_{i=1}^n\tilde z_{ij}^2)$} - \Ep[\tilde z_{ij}^2]\} \frac{\hat\pi}{(1-\hat\pi)^2} \right.\\
\left.  + \max_{j\leq p} \frac{1}{\sqrt{n}}\sum_{i=1}^n g_i \{ \tilde z_{ij}^2 -\Ep[\tilde z_{ij}^2] \} \left\{\frac{\hat\pi}{(1-\hat\pi)^2}-\frac{\pi}{(1-\pi)^2}\right\} \right| \\
  \displaystyle \leq  \max_{j\leq p}| \mbox{$(\frac{1}{n}\sum_{i=1}^n\tilde z_{ij}^2)$} - \Ep[\tilde z_{ij}^2]| \left| \frac{1}{\sqrt{n}}\sum_{i=1}^n g_i \right| \frac{\hat\pi}{(1-\hat\pi)^2}\\
  + \left|\frac{1}{\sqrt{n}}\sum_{i=1}^n g_i \{ \tilde z_{ij}^2 -\Ep[\tilde z_{ij}^2] \} \right| \left|\frac{\hat\pi}{(1-\hat\pi)^2}-\frac{\pi}{(1-\pi)^2}\right| \\
= O_\P(\max_{j\leq p} \mbox{$\{\frac{1}{n}\sum_{i=1}^n\tilde z_{ij}^4\}^{1/2}$})  \sqrt{\frac{\log (pn)}{n}} \{ 1 + 1/\sqrt{p}\}
    \end{array} $$

Therefore, the estimator $\hat\Gamma$ satisfies Condition D and E and we are in position to apply Theorem \ref{thm:mainEst} and Corollary \ref{cor:estmain} provided the other regularity conditions hold.

\section{Numerical Simulations}\label{sec:simulations}

In this section we illustrate the finite sample performance of the inference methodology. We begin with the estimators as described in Algorithm \ref{algo:1} and the pointwise confidence region based on (\ref{region:one}). Then we proceed to investigate the performance of simultaneous confidence regions based on (\ref{region:simultaneous}) and (\ref{def:mb}). Data are simulated from the a high-dimensional linear regression model with error in measurements as described in (\ref{model}). The random variables $\xi_i,$ $x_i,$ and $w_i$ are independent and $\xi_i\sim N(0,\sigma_{\xi}^2),$ $w_i\sim N(0,\sigma^2_{w}I_{p\times p}),$ $x_i\sim  N(0,\Omega)$ where $I_{p\times p}$ is an identity matrix and $\Omega$ is a $p\times p$ matrix with components $\Omega_{ij}=0.5^{|i-j|}.$ We set $\sigma_{\xi}=1,$ and consider three possible choices of $\sigma_{w}=0.25,0.5,1.$ For simplicity, we assume $\sigma_w$ is known. All results are based on 500 replications each with a sample size $n=350.$

We implement the estimation Steps 1 and 2 of Algorithm (\ref{algo:1}) by the conic estimator described in \cite{BRT2014}. The tuning parameters of this estimator are set to $\tau=\mu=\sqrt{\log(p/.05)/n}$ in the notation described in \cite{BRT2014}. The inference results with the proposed algorithm are referred to as `EIV-inference' in the following. To illustrate the effect of measurement error, we also implement Algorithm 1 disregarding measurement error, i.e., estimation Steps 1 \& 2 are implement via ordinary lasso with the observed variables $z,y,$ and Step 3 is implemented assuming $\sigma_w=0.$ These results will be referred to as `Naive inference' in the sequel. All computations are performed in R. High dimensional optimizations are carried out by the package \verb"Rmosek". One dimensional optimization of Step 3 of Algorithm 1 is implemented by the built in `optimize` function. All estimates are truncated at $10^{-7}.$

We examine the two main inference results of this paper. First, inference of a single dimensional target parameter. For this purpose, the first component $\beta_{01}$ of the parameter vector $\beta_0$ is assumed to be the target over which inference is to be performed. The nuisance part of this vector is set to satisfy $\|\beta_{0,-1}\|_0=5$ and all non zero components of this nuisance vector are set to $\beta_{0j}=1$, for  $j \in \{6,7,8,9,10\}.$ We test three cases of the target parameter, $H_0:$ $\beta_{01}=0,$ $\beta_{01}=0.5$ and $\beta_{01}=1.$ We compute the type 1 error: relative frequency of the number of times $H_0$ is rejected when $H_0$ is true and the bias in the estimate. The results are reported in Table \ref{t1}, Table \ref{t2} and Table \ref{t3}.

\begin{table}[ht!]
\centering
\begin{tabular}{lccccccc}
\hline
\multicolumn{2}{l}{\textbf{$H_0\,:\,\beta_{01}=1$}}  & \multicolumn{2}{c}{\textbf{$\sigma_w=1$}} & \multicolumn{2}{c}{\textbf{$\sigma_w=0.5$}} & \multicolumn{2}{c}{\textbf{$\sigma_w=0.25$}} \\ \hline
\textbf{Method}                 & $p$   & \textbf{Size}       & \textbf{Bias}      & \textbf{Size}        & \textbf{Bias}       & \textbf{Size}        & \textbf{Bias}        \\ \hline
\multirow{3}{*}{\textbf{EIV-inference}}   & \textbf{300} & 0.052               & 0.0023             & 0.066                & 0.0096              & 0.058                & 0.0015               \\
                                & \textbf{400} & 0.06                & 0.0134             & 0.058                & 0.0129              & 0.054                & 0.0022               \\
                                & \textbf{500} & 0.068               & 0.0004             & 0.034                & 0.0004              & 0.040                & 0.0067               \\ \hline
\multirow{3}{*}{\textbf{Naive-inference}} & \textbf{300} & 1                   & -0.5377            & 0.992                & -0.3650             & 0.834                & -0.2257              \\
                                & \textbf{400} & 1                   & -0.5349            & 0.996                & -0.3638             & 0.846                & -0.2253              \\
                                & \textbf{500} & 1                   & -0.5439            & 0.996                & -0.3695             & 0.836                & -0.2216              \\ \hline
\end{tabular}
\caption{\small{Simulation results for $H_0:\,\beta_{1}=1.$ For each method we report the type I error (Size) of the corresponding test and average bias (Bias) in the refitted estimates of the target parameter.}}
\label{t1}
\end{table}

\begin{table}[ht!]
\centering
\begin{tabular}{lccccccc}
\hline
\multicolumn{2}{l}{$H_0\,:\,\beta_{01}=0.5$}         & \multicolumn{2}{c}{$\sigma_w=1$}                                       & \multicolumn{2}{c}{$\sigma_w=0.5$}                                       & \multicolumn{2}{c}{$\sigma_w=0.25$}                                       \\ \hline
\textbf{Method}                 & $p$   & \multicolumn{1}{c}{\textbf{Size}} & \multicolumn{1}{c}{\textbf{Bias}} & \multicolumn{1}{c}{\textbf{Size}} & \multicolumn{1}{c}{\textbf{Bias}} & \multicolumn{1}{c}{\textbf{Size}} & \multicolumn{1}{c}{\textbf{Bias}} \\ \hline
\multirow{3}{*}{\textbf{EIV-inference}}   & \textbf{300} & 0.048                             & -0.0048                           & 0.064                             & -0.0008                           & 0.058                             & 0.0021                            \\
                                & \textbf{400} & 0.038                             & 0.0105                            & 0.07                              & -0.0006                           & 0.046                             & -0.0043                           \\
                                & \textbf{500} & 0.036                             & -0.0061                           & 0.082                             & -0.0025                           & 0.060                             & 0.0000                            \\ \hline
\multirow{3}{*}{\textbf{Naive-inference}} & \textbf{300} & 0.932                             & -0.2751                           & 0.666                             & -0.1840                           & 0.312                             & -0.1087                           \\
                                & \textbf{400} & 0.942                             & -0.2698                           & 0.696                             & -0.1837                           & 0.358                             & -0.1147                           \\
                                & \textbf{500} & 0.972                             & -0.2804                           & 0.684                             & -0.1863                           & 0.338                             & -0.1121                           \\ \hline
\end{tabular}
\caption{\small{Simulation results for $H_0:\,\beta_{1}=0.5.$ For each method we report the type I error  (Size) of the corresponding test and average bias (Bias) in the refitted estimates of the target parameter.}}
\label{t2}
\end{table}

\begin{table}[ht!]
\centering
\begin{tabular}{lccccccc}
\hline
\multicolumn{2}{l}{$H_0\,:\,\beta_{01}=0$}         & \multicolumn{2}{c}{$\sigma_w=1$}                                       & \multicolumn{2}{c}{$\sigma_w=0.5$}                                       & \multicolumn{2}{c}{$\sigma_w=0.25$}                                       \\ \hline
\textbf{Method}                 & $p$   & \multicolumn{1}{c}{\textbf{Size}} & \multicolumn{1}{c}{\textbf{Bias}} & \multicolumn{1}{c}{\textbf{Size}} & \multicolumn{1}{c}{\textbf{Bias}} & \multicolumn{1}{c}{\textbf{Size}} & \multicolumn{1}{c}{\textbf{Bias}} \\ \hline
\multirow{3}{*}{\textbf{EIV-inference}}   & \textbf{300} & 0.042                             & 0.0178                            & 0.054                             & 0.0046                            & 0.054                             & 0.0058                            \\
                                & \textbf{400} & 0.040                             & 0.0138                            & 0.046                             & 0.0211                            & 0.066                             & 0.0202                            \\
                                & \textbf{500} & 0.052                             & -0.0047                           & 0.054                             & -0.0003                           & 0.046                             & 0.0029                            \\ \hline
\multirow{3}{*}{\textbf{Naive-inference}} & \textbf{300} & 0.060                             & 0.0123                            & 0.050                             & 0.0044                            & 0.050                             & 0.0041                            \\
                                & \textbf{400} & 0.060                             & 0.0055                            & 0.048                             & 0.0136                            & 0.068                             & 0.0151                            \\
                                & \textbf{500} & 0.066                             & 0.0004                            & 0.054                             & 0.0006                            & 0.050                             & 0.0026                            \\ \hline
\end{tabular}
\caption{\small{Simulation results for $H_0:\,\beta_{1}=0.$ For each method we report the type I error  (Size) of the corresponding test and average bias (Bias) in the refitted estimates of the target parameter.}}
\label{t3}
\end{table}

Next we illustrate the construction of simultaneous confidence regions for a multi-dimensional parameter vector. Here we test the hypothesis $H_0:\beta_{0k}=0, k\in\{1,\ldots,10\}$. The parameter vector is set to satisfy $\|\beta_0\|_0=5$ with the non zero components set to $\beta_{0j}=1$ for  $j \in \{16,17,18,19,20\}.$ In Table \ref{t4} we report the family wise error rate (FWER), relative frequency of the number of times $H_0$ is rejected when $H_0$ is true.

Numerical results support our theoretical findings. The proposed estimator based on Algorithm \ref{algo:1} (`EIV-inference') provides control on the type 1 error of the test at the significance level of 0.05 at all considered settings. In comparison, disregarding measurement error leads to a severely inflated type 1 error. We also observe the classical effect of disregarding measurement error in the refitted estimates, i.e., estimates are biased towards zero. This is illustrated by a negative bias in the case of `Naive-inference.' Lastly, for the case of inference over multiple target parameters, `EIV-inference' provides the desired control on the FWER.

\begin{table}[ht!]
\centering
\label{my-label}
\begin{tabular}{ccccc}
\hline
$n=350$ & \multicolumn{1}{c}{\textbf{$\sigma_w=1$}}                                       & \multicolumn{1}{c}{\textbf{$\sigma_w=0.5$}}                                     \\ \hline
$p$     & \multicolumn{1}{c}{\textbf{FWER}}  & \multicolumn{1}{c}{\textbf{FWER}}  \\ \hline
\textbf{300}     & 0.058                                     & 0.040                                                    \\
\textbf{400}     & 0.046                                        & 0.058                                                         \\ \hline
\end{tabular}
\caption{\small{Simulation results of `EIV-inference' for $H_0:\,\beta_{0k}=0, k\in\{1,\ldots,10\}.$ For each method we report the family wise error rate (FWER).}}
\label{t4}
\end{table}


\begin{appendix}

\section{Proofs of Section 3}

\begin{proof}[Proof of Theorem \ref{thm:main}]

The linear representation result follows from  Theorem \ref{thm:mainEst} with $\hat\Gamma = \Gamma$ so that $\|\Gamma-\hat \Gamma\|_\infty =0$.

\end{proof}

\begin{proof}[Proof of Corollary \ref{thm:Linfty}]
For notational convenience we use $\En[\cdot]=\frac{1}{n}\sum_{i=1}^n[\cdot]_i$. By Theorem \ref{thm:main} we have for all $1\leq j\leq p$
$$ \sqrt{n}\sigma_j^{-1}(\check\beta_j-\beta_{0j}) = \frac{1}{\sqrt{n}}\sum_{i=1}^n \bar\psi_j(y_i,z_i) + O_\P(\delta_n).$$
Since $\delta_n \to 0$ given in Condition C, with probability $1-o(1)$ we have
$$\frac{\|\check\beta - \beta_0\|_\infty}{\max_{j\leq p}\sigma_j}  \leq C\max_{j\leq p}\frac{1}{n}\sum_{i=1}^n \bar\psi_j(y_i,z_i) + O(\delta_n^{1/2}/\sqrt{n})$$
Since $\Ep[ \bar\psi_j(y,z)]=0$ and $\Ep[ \bar\psi_j^2(y,z) ]=1$, defining the event $$\mathcal{E}_{y,z}=\{\max_{j\leq p}\En[\bar\psi_j^2(y,z)]^{1/2}\leq 2\}$$ we have by symmetrization (Lemma 2.3.7 of \cite{van1996weak}), provided $t\geq 4/\sqrt{n}$,
$$\begin{array}{rl}
P\left( \max_{j\leq p}\En[\bar\psi_j(y,z)] > t \right) & \leq 4 P\left( \max_{j\leq p}\En[ r\bar\psi_j(y,z)] > t/4 \right)\\
& \leq 4 \Ep[ P\left( \max_{j\leq p}\En[r\bar\psi_j(y,z)] > t/4 \mid \mathcal{E}_{y,z}\right)] \\
& + 4P\left( \mathcal{E}_{y,z}^c \right) \\
\end{array}
$$
where $(r_i)_{i=1}^n$ are i.i.d. Rademacher random variables independent of the data. Since $\Ep[\En[r^2\bar\psi_j^2(y,z)\mid (y_i,z_i)_{i=1}^n]=\En[\bar\psi_j^2(y,z)]$, and conditionally on $\mathcal{E}_{y,z}$
$$\begin{array}{rl}
\Ep[\max_{j\leq p}\left|\En[r\bar\psi_j(y,z)]\right| \mid (y_i,z_i)_{i=1}^n] & \leq C \sqrt{\frac{\log 2p}{n}} \max_{j\leq p} \En[\bar\psi_j^2(y,z)]^{1/2} \\
& \leq 2C \sqrt{\frac{\log 2p}{n}}\end{array}$$ by Corollary 2.2.8 in \cite{vdV-W}.
Therefore, for $t/4:= C\sqrt{\frac{\log 2p}{n}} + 2\sqrt{\frac{2\log(1/\varepsilon)}{n}}$ we have that
$$\begin{array}{rl}
P\left( \max_{j\leq p}\En[\bar\psi_j(y,z)] > t \right) & \leq 4 P\left( \max_{j\leq p}\En[ r\bar\psi_j(y,z)] > t/4 \right)\\
& \leq 4 \varepsilon + 4P\left( \mathcal{E}_{y,z}^c \right) \\
& = 4\varepsilon + o(1)\\
\end{array}
$$
where the last line follows from Step 2 below. Therefore, with probability $1-o(1)$
\begin{equation}\label{bound:sum_psi_bar}
\max_{1\leq j\leq p} \left|\frac{1}{n}\sum_{i=1}^n\bar\psi_j(y_i,z_i) \right| \leq C\sqrt{\frac{\log(pn)}{n}}
\end{equation}

The result follows by bounding $ \max_{j\leq p}\sigma_j $ from above. Indeed we have
$ \max_{j\leq p}\sigma_j \leq C (1 + \|\beta_0\|) $ since $\sigma_j^2=\Sigma_j^{-2}\Ep[\psi_j^2(y,z)]$ where $\Sigma_j$ is bounded away from zero by $\Ep[xx^T]$ having eigenvalues bounded away from zero by Condition A, and $\Ep[\psi_j^2(y,z)]\leq C(1+\|\beta_0\|)^2$ since $\|\Gamma\|_{op}$ is bounded and $\max_{j\leq p}\|\mu_0^j\| \leq C'$. Indeed we have
$$\begin{array}{rl}
\Ep[x_j^2] & \geq \Ep[ (x_j-x_{-j}^T\mu^j_0)^2] = (e^j-\mu^j_0)^T \Ep[xx^T] (e^j-\mu^j_0) \\
& \geq \|\mu^j_0\|^2 {\rm mineig}(\Ep[xx^T])\end{array}$$ since the support of $e^j$ and $\mu^j_0$ is disjoint by construction. Therefore, \begin{equation}\label{bound:muL2norm}
\|\mu^j_0\| \leq  \{ \Ep[x_j^2] / {\rm mineig}(\Ep[xx^T]) \}^{1/2} \leq C'.\end{equation}

Step 2. Since $\Ep[\bar\psi_j^2(y,z)]=1$ we have
$$\begin{array}{rl}
P(\mathcal{E}_{y,z})& = P( \max_{j\leq p}\En[\bar\psi_j^2(y,z)] \leq 4 ) \\
& \leq P( \max_{j\leq p}|\En[\bar\psi_j^2(y,z)]-\Ep[\bar\psi^2(y,z)]| \geq 3 ) \\
& \leq \Ep[\max_{j\leq p}|\En[\bar\psi_j^2(y,z)]-\Ep[\bar\psi^2(y,z)]|]\\
& \leq C \frac{M_2\log p}{n}+ C\sqrt{\frac{\log p}{n}}M^{1/2}_2
\end{array} $$
where we used Markov's inequality and Lemma \ref{lem:m2bound} with $X_{ij}=\bar\psi_j(y_i,z_i)$ and $k=2$. In this case, by Lemma \ref{lem:subgaussian} we have $$\begin{array}{rl}
M_2 & = \Ep[ \max_{i\leq n,j\leq p} \sigma_j^{-2}\Sigma_j^{-2} |(\xi_i-w_i^T\beta_0)(e^j-\mu^j_0)^Tz_i+ \Gamma_{jj} \beta_{0j})|^2] \\
& \leq C(1+\|\beta_0\|)^2\log (n) \max_{j\in S}\sigma_j^{-2}\Sigma_j^{-2}(1+\|\mu^j_0\|)^2 \log(pn)\\
 & + C\max_{j\in S}\sigma_j^{-2}\Sigma_j^{-2}|\Gamma_{jj} \beta_{0j}|^2\\
& \leq C'(1+\|\beta_0\|)^2\log (n) \max_{j\in S}\sigma_j^{-2}\Sigma_j^{-2}(1+\|\mu^j_0\|)^2 \log(pn) \end{array}$$
where we used that $\max_{j\in S}\sigma_j^{-2}\Sigma_j^{-2}|\Gamma_{jj} \beta_{0j}|^2\leq C\|\beta_0\|^2$ under Condition A. Thus, provided that $(1+\|\beta_0\|) \log (pn) \log^{1/2} n \leq \delta_n \sqrt{n}$ where $\delta_n\to 0$, we have $P(\mathcal{E}_{y,z})=o(1)$.

\end{proof}

\begin{proof}[Proof of Corollary \ref{cor:PointwiseInference}]
The result follows directly from the linear representation of Theorem \ref{thm:main} combined with an application of the Lyapunov central limit theorem since $\bar\psi_j(y_i,z_i)$, $i=1,\ldots,n$, are independent random variables with  $\Ep[\bar\psi_j(y,z)]=0$ and $\Ep[\bar\psi_j^2(y,z)]=1$. To verify the last condition, note that for any $m\geq 2$, we have
\begin{equation}\label{barpsi:mmoment}\begin{array}{rl}
\Ep[|\bar\psi_j(y,z)|^m] & \leq \tilde C^m \{ \Ep[|z_j\xi|^m]+\Ep[|z_jw^T\beta_0|^m]+\|\Gamma\|_\infty^m\|\beta_0\|^m\\
& +\Ep[|z_{-j}^T\mu_0^j\xi|^m]+\Ep[|z_{-j}^T\mu_0^jw^T\beta_0|^m]\\
& +\|\mu^j_0\|^m \|\beta_0\|^m  \|\Gamma\|_\infty^m\}\\
& \leq (1+C)^m \tilde C^m C_{2m}^{2m} (1+\|\beta_0\|)^m(1+\|\mu^j_0\|)^m
\end{array}\end{equation} since $\|\Gamma\|_\infty\leq C$ by Condition A, and by Lemma \ref{lem:subgaussian}(1) since $z$, $w$, and $\xi$ are subgaussian random variables,  where $C_{2m} = C' \sqrt{2m}$ for some universal $C'$. Therefore, by applying Lemma \ref{lem:auxmoment} with $k=3$, $B_m= (1+C)\tilde CC_{2m}^2$ and $\gamma=(1+\|\beta_0\|)(1+\|\mu^j_0\|)$, we have $\Ep[|\bar\psi_j(y,z)|^3]\leq B_{2(1+\delta)/\delta}^{1+\delta}(1+\|\beta_0\|)^{1+\delta}(1+\|\mu^j_0\|)^{1+\delta}$ for any $\delta \in (0,2]$ where $B_{2(1+\delta)/\delta}^{1+\delta} = \{ (1+C) \tilde CC_{4(1+\delta)/\delta}^2\}^{(1+\delta)} \leq C'' (1/\delta)^{(1+\delta)}$.
Thus we have by Condition C that
$$\begin{array}{rl}
 \frac{1}{n^{3/2}}\sum_{i=1}^n\Ep[|\bar\psi_j(y,z)|^3] & \leq \frac{B_{2(1+\delta)/\delta}^{1+\delta}(1+\|\beta_0\|)^{1+\delta}(1+\|\mu^j_0\|)^{1+\delta}}{\sqrt{n}}\\
 & =o\left( \frac{(1+\|\beta_0\|)^{\delta}(1+\|\mu^j_0\|)^{\delta}}{\delta^{(1+\delta)}\log(pn)}\right)\end{array}$$
It suffices to show that the argument in the little-$o$ term is bounded for a suitable choice of $\delta\to 0$. Condition C also implies that $\|\beta_0\|\leq C\sqrt{n}$ and  (\ref{bound:muL2norm}) implies $\|\mu^j_0\|\leq C$. Therefore it suffices to show that for some $\delta \to 0$ we have
$ n^{\delta/2} \leq C$ and $\delta^{(1+\delta)}\log(pn) \geq c$. Indeed, setting $\delta = 1/\log n$ we have $n^{\delta/2}=e^{1/2}$ and $\delta^{(1+\delta)}\log(pn) \geq  \{ \log^{1+ 1/\log n} n\}^{-1} \log n  =  1/\{ \log n\}^{1/\log n} \geq e^{-1}$ for $n\geq e^e$.

\end{proof}

Next we show that the distribution of the maximum estimation error is close to the distribution of the maximum of the entries of (a sequence of) Gaussian vectors. Let $(\mathcal{G}_j)_{j\in S}$ denote a tight zero-mean Gaussian vector whose dimension could grow with $n$. Its covariance matrix is given by $(\Ep[\bar \psi_j(y,z)\bar\psi_k(y,z)])_{j,k}, j,k\in S$. We have the following lemma.
\begin{lemma}\label{cor:gaussapprox}
Suppose that Conditions A, B and C hold,  $S\subseteq\{1,\ldots,p\}$, $|S|\geq 2$, and
\begin{itemize}
\item[(i)] $\max_{j\in S} \{\sigma_j^{-1}\Sigma_{j}^{-1}(1+\|\beta_0\|)(1+\|\mu_0^j\|)\}^4\log^7(|S|) = o(n)$, and
\item[(ii)] $\max_{j\in S} \sigma_j^{-1}\Sigma_{j}^{-1}\sqrt{\log |S|}\delta_n = o(1)$
\end{itemize}
where $\delta_n$ is as defined in Condition C. Then we have that
$$ \sup_{t \in \mathbb{R}} \left|P\left(\max_{j\in S}|\sqrt{n}\sigma_j^{-1}(\check\beta_{0j}-\beta_{0j})|\leq t\right)- P\left(\max_{j\in S}|\mathcal{G}_j|\leq t\right) \right|=o(1)$$
\end{lemma}
\begin{proof}[Proof of Lemma \ref{cor:gaussapprox}]

Define the following random variables
 $$\begin{array}{rl}
 Z &= \max_{j\in S} |\sqrt{n}\sigma_j^{-1}(\check\beta_j - \beta_{0j})|,\\
 \bar Z &= \max_{j\in S} \left|\frac{1}{\sqrt{n}}\sum_{i=1}^n\bar\psi_j(y_i,z_i)\right|, \ \ \mbox{and}\\
 \widetilde Z &= \max_{j\in S} |\mathcal{G}_j|\end{array}$$ where $(\mathcal{G}_j)_{j\in S}$ is a zero mean Gaussian vector with covariance matrix given by $C_{jk}=\Ep[\bar\psi_j(y,z)\bar\psi_k(y,z)]$ for $k,j\in S$.

By the triangle inequality  we have
$$\begin{array}{rl}
|Z - \widetilde Z| & \leq |Z-\bar Z| + |\bar Z - \widetilde Z| \\
& \leq \max_{j \in S} \left| \sqrt{n}\sigma_j^{-1}(\check\beta_j - \beta_{0j}) -\frac{1}{\sqrt{n}}\sum_{i=1}^n\bar\psi_j(y_i,z_i) \right| + |\bar Z - \widetilde Z| \\
& = O_\P(\delta_n\max_{j\in S} \sigma_j^{-1}\Sigma_j^{-1})  + |\bar Z - \widetilde Z|
\end{array}
$$ where the last step follows from Theorem \ref{thm:main}.

Next to bound $|\bar Z - \widetilde Z|$ we will apply Theorem 3.1 in \cite{chernozhukov2015noncenteredprocesses} specialized to the i.i.d. setting. Let $X_{ij}=\bar\psi_j(y_i,z_i)$, $\Ep[X_{ij}]=0$, $Y_i\sim N(0,\Ep[X_iX_i'])$, so that  $\bar Z = \max_{j\in S} |n^{-1/2}\sum_{i=1}^nX_{ij}|$ and $\widetilde Z =  \max_{j\in S} |n^{-1/2}\sum_{i=1}^nY_{ij}|$. Define
$$\begin{array}{ll}
L_n=\max_{j\in S}\Ep[|X_{1j}|^3]\\
M_{n,X}(\delta) = \Ep[ \max_{j\in S}|X_{1j}|^31\{ \max_{j\in S}|X_{1j}|> \delta \sqrt{n}/\log |S| \} ], \ \mbox{and}\\
M_{n,Y}(\delta) = \Ep[ \max_{j\in S}|Y_{1j}|^31\{ \max_{j\in S}|Y_{1j}|> \delta \sqrt{n}/\log |S| \}]\end{array}$$
By Theorem 3.1 in \cite{chernozhukov2015noncenteredprocesses}, we have that for every Borel set $A\subseteq \mathbb{R}$ such that
$$\begin{array}{rl} P ( \max_{j\in S} n^{-1/2}\sum_{i=1}^nX_{ij} \in A ) & \leq P(\max_{j\in S} n^{-1/2}\sum_{i=1}^nY_{ij} \in A^{C\delta} )\\
 & +  C' \frac{\log^2|S|}{\delta^3\sqrt{n}}\{L_n+M_{n,X}(\delta)+M_{n,Y}(\delta)\}\end{array}$$
where $A^{C\delta}=\{ t \in \mathbb{R} : {\rm dist}(t,A)\leq C\delta\}$ is an $C\delta$-enlargement of $A$, and  $|S|\geq 2$. We proceed to bound $L_n$, $M_{n,X}(\delta)$ and $M_{n,Y}(\delta)$. For notational convenience define $m_j=\sigma_j^{-1}\Sigma_j^{-1}(1+\|\beta_0\|)(1+\|\mu^j_0\|)$ and denote $m_S:=\max_{j\in S}m_j$.

Note that $\bar\psi_j(y_1,z_1)=\sigma_j^{-1}\Sigma_j^{-1}\psi_j(y_1,z_1)$ where $\psi_j(y_1,z_1)$ is the product of subgaussian random variables with parameters $C'(1+\|\beta_0\|)$ and $C'(1+\mu^j_0\|)$. By (\ref{barpsi:mmoment}), for $k\geq 3$ we have
$$ \Ep[|\bar\psi_j(y_1,z_1)|^k]^{1/k} \leq C k \sigma_j^{-1}\Sigma_j^{-1}(1+\|\beta_0\|)(1+\|\mu^j\|)\leq C k m_j.$$
Since $\Ep[\bar\psi_j^2(y_1,z_1)]=1$, by Lemma \ref{lem:auxmoment} (with $k=3, \delta=1$),  we have $$L_n=\max_{j\in S}\Ep[|X_{1j}|^3] = \max_{j\in S}\Ep[|\bar\psi_j(y_1,z_1)|^3] \leq C'm_S^2.$$

To bound $M_{n,X}(\delta)$ we have that
$$\begin{array}{rl}
M_{n,X}(\delta) & = \Ep[ \max_{j\in S}|X_{1j}|^3 1\{ \max_{j\in S}|X_{1j}|> \delta \sqrt{n}/\log |S| \} ]\\
 & \leq \Ep[ \max_{j\in S}|X_{1j}|^6]^{1/2}\{ P( \max_{j\in S}|X_{1j}| > \delta \sqrt{n}/\log |S| )\}^{1/2}\\
 & \leq C\{m_S\log |S|\}^{3} \{ P( \max_{j\in S}|X_{1j}| > \delta \sqrt{n}/\log |S| )\}^{1/2}\\
 & \leq C\{m_S\log |S|\}^{3} \{ 2|S| \exp( -c\{\delta \sqrt{n}/\log |S|\}/m_S)\}^{1/2}\\
& \leq  C
 \end{array}$$
where the first inequality follows by Cauchy--Schwartz, the second by Lemma \ref{lem:auxmoment}, the third by Lemma \ref{aux:barpsi}, and the last step holds provided that $\delta \geq 2\frac{m_S}{c} n^{-1/2} \log^2 (2|S|)  \log(\{m_S\log |S|\}^{3})$.

Next recall that $\max_{j\in S}\Ep[Y_{1j}^2]=1$ and we have
$$\begin{array}{rl}
M_{n,Y}(\delta) & = \Ep[\max_{j\in S}|Y_{1j}|^3 1\{ \max_{j\in S}|Y_{1j}|> \delta \sqrt{n}/\log |S| \} ]\\
& \leq  \{\Ep[\max_{j\in S}|Y_{ij}|^6]\}^{1/2} \{ \Ep[1\{ {\displaystyle \max_{j\in S}}|Y_{1j}|> \frac{\delta \sqrt{n}}{\log |S|} \}]^{1/2}\\
& \leq C \log^{3/2}(|S|C') \{2|S|\exp( -\frac{\delta^2n}{2\log^2|S|}) \}^{1/2}\\
& \leq  C
 \end{array}$$
where the last step holds provided that $\delta > 8n^{-1/2} \log(|S|)  \{\log^{1/2}(|S|)+\log^{1/2}(\log^{3/2}(|S|C'))\}$.

Next note that if $\delta \geq  \frac{L_n^{1/3}\log^{2/3}(|S|)}{\gamma^{1/3}n^{1/6}}$ we have $\log^2|S|/\{\delta^3n^{1/2}\} \leq \gamma/L_n\leq \gamma$ since $L_n \geq 1$. Recall that $\max_{j\in S}m_j\leq n$ for $n$ sufficiently large. Therefore, by setting $\delta:=  \frac{L_n^{1/3}\log^{2/3}(|S|)}{\gamma^{1/3}n^{1/6}} + 8\frac{m_S}{1\wedge c} \frac{\log^2 (2|S|)}{n^{1/2}}  \log(\{m_S\log |S|\}^{3})$, we have
$$ P ( |\bar Z - \widetilde Z| > C\delta ) \leq  C' \gamma.$$

By Strassen's Theorem, there exists a version of $\widetilde Z$ such that
$$
\left| \bar Z - \widetilde Z \right| = O_\P\left( \frac{m_S^{2/3}\log^{2/3}(|S|)}{n^{1/6}} + m_S \frac{\log^2 (2|S|)}{n^{1/2}}  \log(m_S\log |S|) \right)
$$
The result then follows from Lemma 2.4 in \cite{chernozhukov2012gaussian} so that $\sup_{t\in \mathbb{R}}|P( Z \leq t ) -   P(\widetilde Z \leq t )|  = o(1)$ by noting that $\Ep[\widetilde Z] \leq C\sqrt{\log |S|}$ since $\Ep[\mathcal{G}_j^2]=1$, provided that
{\small $$\sqrt{\log |S|} \left(  \delta_n\max_{j\in S} \sigma_j^{-1}\Sigma_j^{-1} +  \frac{m_S^{2/3}\log^{2/3}|S|}{n^{1/6}} +  \frac{m_S\log^2 (|S|)\log(m_S\log |S|)}{n^{1/2}}  \right)=o(1),$$}
which holds under condition (i), namely $m_S^4\log^7(|S|)=o(n)$, and condition (ii) on $\delta_n$. Indeed, the first term is controlled by (ii), the second term by (i), and for the third term note that $m_S\log^{5/2}(|S|)\log(m_S\log |S|)=o(n^{1/2})$ is equivalent to $$m_S^2\log^5(|S|)\log^2(m_S\log |S|)=o(n)$$ which is implied by (i).
\end{proof}

\begin{proof}[Proof of Theorem \ref{thm:inference}]
We divide the proof in steps. Step 1 is the main argument which invokes the other steps for auxiliary calculations. Let $c_\alpha^*$ denote the $(1-\alpha)$-conditional quantile of $\widetilde Z^* = \max_{j\in S} |\widehat{\mathcal{G}}_j|$ given the data $(y_i,z_i)_{i=1}^n$ and
 $c_\alpha^0$ denote the $(1-\alpha)$-conditional quantile of $\widetilde Z= \max_{j\in S} |{\mathcal{G}}_j|$ where $\mathcal{G}$ is the Gaussian random vector defined in Lemma \ref{cor:gaussapprox}. For some $\vartheta_n \to 0$, we have that
$$ \begin{array}{rl}
&\displaystyle \Pr{\max_{j\leq S}|\sqrt{n}\hat\sigma_j^{-1}(\check\beta_{0j}-\beta_{0j})|\leq c_\alpha^*} \\
& = \Pr{ |\sqrt{n}\sigma_j^{-1}(\check\beta_{0j}-\beta_{0j})|\leq  (\hat\sigma_j/\sigma_j)c_\alpha^*, \forall j\in S} \\
& \leq_{(a)}  \Pr{ \max_{j\in S}|\sqrt{n}\sigma_j^{-1}(\check\beta_{0j}-\beta_{0j})|\leq  (1+\varepsilon_n)c_\alpha^*}+o(1)\\
& \leq_{(b)}  \Pr{ \max_{j\in S}|\sqrt{n}\sigma_j^{-1}(\check\beta_{0j}-\beta_{0j})|\leq c^0_{\alpha-\vartheta_n}}+o(1)\\
& \leq_{(c)}  \Pr{ \max_{j\in S}|\mathcal{G}_j|\leq c^0_{\alpha-\vartheta_n}}+o(1)\\
& = 1-\alpha+\vartheta_n+o(1)
\end{array}$$
where (a) follows from $|\hat\sigma_j/\sigma_j|\leq (1+\varepsilon_n)$ with probability $1-o(1)$ by Step 2 below, (b) follows from $c_\alpha^* (1+\varepsilon_n) \leq c^0_{\alpha-\vartheta_n}$ with probability $1- o(1)$ by Step 3 below, (c) follows by Lemma \ref{cor:gaussapprox}, and the last step by definition of $c^0_{\alpha-\vartheta_n}$.

The other inequality follows similarly.

~\\

Step 2. In this step we prove the claim: \\
For some $\varepsilon_n = O(\delta_n\log^{1/2}(n|S|)(1+\|\beta_0\|)(1+\max_{j\in S}\|\mu^j_0\|))$ we have \begin{equation}\label{step2:main}P\left( \max_{j\in S} |\sigma_j/\hat\sigma_j|\vee |\hat \sigma_j/\sigma_j|\leq 1+\varepsilon_n\right) \geq 1-o(1)\end{equation}
Let $\widetilde\psi_j(y,z)=\psi_j(y,z,\check\beta_{0j},\hat\eta^j)$ and $\psi_j(y,z)=\psi_j(y,z,\beta_{0j},\eta^j_0)$, so that $\hat\sigma_j=\hat\Sigma_j^{-1}\{\hat\En[\widetilde\psi_j^2(y,z)]\}^{1/2}$ and  $\sigma_j=\Sigma_j^{-1}\{\Ep[\psi_j^2(y,z)]\}^{1/2}$. Thus, we have that
\begin{equation}\label{eq:sigma:approx}
\begin{array}{rl}
|\hat\sigma_j - \sigma_j| & \leq \hat\Sigma_j^{-1}|  \En[\widetilde\psi^2_j(y,z)]^{1/2} - \En[\psi^2_j(y,z)]^{1/2}|\\
& + \hat\Sigma_j^{-1}| \En[\psi^2_j(y,z)]^{1/2} -  \Ep[\psi^2_j(y,z)]^{1/2}| \\
 & + |\hat\Sigma_j^{-1}-\Sigma_j^{-1}|\Ep[\psi^2(y,z)]^{1/2}\\
& \leq \hat\Sigma_j^{-1}\En[\{\widetilde\psi_j(y,z)-\psi_j(y,z)\}^2]^{1/2}\\
&+ \hat\Sigma_j^{-1}\frac{|\En[\psi^2_j(y,z)] -  \Ep[\psi^2_j(y,z)]| }{\En[\psi^2_j(y,z)]^{1/2} +  \Ep[\psi^2_j(y,z)]^{1/2}} + \frac{|\Sigma_j-\hat\Sigma_j|}{\hat\Sigma_j\Sigma_j}\Ep[\psi^2_j(y,z)]^{1/2}\\
& \leq \hat\Sigma_j^{-1}\En[\{\widetilde\psi_j(y,z)-\psi_j(y,z)\}^2]^{1/2}\\
&+ \hat\Sigma_j^{-1}\sigma_j\Sigma_j\frac{|\En[\psi^2_j(y,z)] -  \Ep[\psi^2_j(y,z)]| }{\Ep[\psi^2_j(y,z)]} + \hat\Sigma_j^{-1}|\Sigma_j-\hat\Sigma_j| \sigma_j
\end{array}
\end{equation}

To bound the first term in the RHS of (\ref{eq:sigma:approx}), using the notation $\widetilde\psi_j(y,z)=\psi_j(y,z,\check\beta_{0j},\hat\eta^j)$ and $\psi_j(y,z)=\psi_j(y,z,\beta_{0j},\eta^j_0)$, and the triangle inequality we have
\begin{equation}\label{Aux:L2Pn}
 \begin{array}{rl}
 & \En[\{\widetilde\psi_j(y,z)-\psi_j(y,z)\}^2]^{1/2}  \\
 & \leq \En[\{z_{-j}^T(\hat\mu^j-\mu^j_0)\}^2\{y-z^T\beta_0\}^2]^{1/2}\\
& + \En[\{z^T(\hat\beta-\beta_0)\}^2\{z_j-z_{-j}^T\mu^j_0\}^2]^{1/2}\\
& + \En[\{z_{-j}^T(\hat\mu^j-\mu^j_0)\}^2\{z^T(\hat\beta-\beta_0)\}^2]^{1/2}\\
& + |(e^j-\mu^j_0)^T\Gamma(\hat\beta-\beta_0)|+|(\hat\mu^j-\mu^j_0)^T\Gamma \beta_0|\\
&+|(\hat\mu^j-\mu^j_0)^T\Gamma (\hat\beta-\beta_0)|\\
& \leq C\max_{j\in S}(1+\|\beta_0\|)(1+\|\mu^j_0\|)\sqrt{\frac{s\log(pn)}{n}}\log^{1/2}(|S|n) \\
& = O(\delta_n/\sqrt{s})
\end{array}\end{equation}
where we used that $\max_{i\leq n, j\in S}|y_i-z_i^T\beta_0| \leq C(1+\|\beta_0\|)\log^{1/2}(n)$, and  $\max_{i\leq n, j\in S}|z_{ij}-z_{i,-j}^T\mu_0^j| \leq C\max_{j\in S}(1+\|\mu^j_0\|)\log^{1/2}(n|S|)$ with probability $1-o(1)$ by setting $C$ large enough constant since $y_i$ and $z_{i}$ are subgaussian random vectors,
from the rates of convergence and sparsity assumptions in Condition B combined with Lemma \ref{thm:KLthm4} (which establishes that $Cs$-sparse eigenvalues are bounded above by a constant), $\|\Gamma\|_{op}=\|\Gamma\|_\infty \leq C$, and the last line follows from the requirement of $\delta_n$ in Condition C.

To bound the second term in the RHS of (\ref{eq:sigma:approx}), we will apply Lemma \ref{lem:m2bound} with $X_{ij}:= \psi_j(y_i,z_i)/\Ep[\psi_j^2(y,z)]^{1/2}, j\in S$, $p=|S|$ (assumed $|S|\geq 2$ for convenience) and $k=2$. Noting that $\Ep[\psi_j^2(y,z)]^{1/2} = \sigma_j\Sigma_j$, we have
{ $$ \begin{array}{rl}
&\Ep\left[  \max_{j\in S}\frac{|\En[\psi^2_j(y,z)]-\Ep[\psi_j^2(y,z)]|}{\Ep[\psi_j^2(y,z)]} \right]  \lesssim \sqrt{\frac{\log |S| }{n}} M_{k}^{1/2}\\
& \lesssim \frac{\log^{1/2}(|S|) \log(pn)}{\sqrt{n}}{\displaystyle\max_{j\in S}}(\sigma_j\Sigma_j)^{-1}(1+\|\beta_0\|)(1+\|\mu^j_0\|) \\
&  \lesssim \frac{\delta_n}{s} \log^{1/2}(|S|)\max_{j\in S}(\sigma_j\Sigma_j)^{-1}
\end{array}$$}

\noindent where we used that $M_{k} \lesssim \max_{j\in S}\sigma_j^{-2}\Sigma_j^{-2}(1+\|\beta_0\|)^2(1+\|\mu^j_0\|)^2 \log^2(pn)$ by the subgaussian assumption in Condition A and Lemma \ref{lem:subgaussian}, and Condition C. Thus by Lemma \ref{lemma:CCK} with $t=\log n$, $q=6$, and $|\mathcal{F}|=|S|$, and using that $\sigma_j\Sigma_j$ is bounded away from zero, we have that with probability $1-o(1)$
$$\begin{array}{l}
 {\displaystyle \max_{j\in S} }\frac{|\En[\psi^2_j(y,z)]-\Ep[\psi_j^2(y,z)]|}{\Ep[\psi_j^2(y,z)]}  \leq C \frac{\delta_n}{s}\log^{1/2}(|S|)\max_{j\in S}(\sigma_j\Sigma_j)^{-1} \\
 + Cn^{-1/2}\{1+ n^{-1/2}(1+\|\beta_0\|)(1+\max_{j\in S}\|\mu^j_0\|)\log (|S|n)\}\log^{1/2}n\\
 +Cn^{-1}(1+\|\beta_0\|)(1+\max_{j\in S}\|\mu^j_0\|)\log^{2}(n|S|)\\
 \lesssim \delta_n \log^{1/2}(|S|)  + Cn^{-1/2}\log(n|S|)
 \end{array}$$
under $n^{-1/2}(1+\|\beta_0\|)(1+\max_{j\in S}\|\mu^j_0\|)\log (|S|n)=o(1)$ implied by Condition C. Note further that $\delta_n \geq n^{-1/2} \log(pn)$ so the last term is negligible compared with the first.

To bound the last term in the RHS of (\ref{eq:sigma:approx}), we have with probability $1-o(1)$
$$\begin{array}{rl}
&\max_{j\in S}|\hat\Sigma_j-\Sigma_j|\sigma_j \\
& \leq \max_{j\in S}|\En[z_jz_{-j}^T](\hat\mu^j-\mu_0^j)|\sigma_j \\
&+  \max_{j\in S}|\En[(z_j-z_{-j}^T\mu^j_0)z_j]-\Ep[(z_j-z_{-j}^T\mu^j_0)z_j]|\sigma_j\\
& \leq \max_{j\in S}\{\En[z_j^2]\}^{1/2}\{\En[\{z_{-j}^T(\hat\mu^j-\mu_0^j)\}^2]\}^{1/2}\sigma_j \\
&+  \max_{j\in S}|\En[(z_j-z_{-j}^T\mu^j_0)z_j]-\Ep[(z_j-z_{-j}^T\mu^j_0)z_j]|\sigma_j\\
& \lesssim  {\displaystyle\max_{j\in S}}\sigma_j(1+\|\mu^j_0\|)\sqrt{\frac{s\log(pn)}{n}}+ {\displaystyle \max_{j\in S}}\sigma_j(1+\|\mu^j_0\|)\sqrt{\frac{\log(|S|n)}{n}}\\
& \lesssim \delta_n /\sqrt{s\log(pn)}\end{array}$$
\noindent where the first step follows from  the triangle inequality, the second step from Cauchy-–Schwarz inequality, the third from the rates of convergence and sparsity assumptions in Condition B (and noting that sparse eigenvalues of  order $Cs$ of $\En[zz^T]$ are bounded above with probability $1-o(1)$ by Lemma \ref{thm:KLthm4}), $\max_{j\in S}\{\En[z_j^2]\}^{1/2}\leq C$ with probability $1-o(1)$, and Lemma \ref{lem:extra}. The last step follows from Condition C and $\sigma_j \leq C(1+\|\beta_0\|)$.

Finally, note that  $\Sigma_j$ is bounded away from zero and from above so that $\hat\Sigma_j$ is also bounded from below uniformly in $j\in S$ and $n$ with probability $1-o(1)$.

Combining these relations we have $$\varepsilon_n \leq C\delta_n\log^{1/2}(|S|)(1+\|\beta_0\|)(1+\max_{j\in S}\|\mu^j_0\|).$$


~\\

Step 3. In this step we show that there is a sequence $\vartheta_n \to 0$ such that
$$ P( c_\alpha^*(1+\varepsilon_n) > c^0_{\alpha-\vartheta_n}) = o(1)$$
where $\varepsilon_n$ is defined in Step 2.

Recall that $\widehat{\mathcal{G}}_j=\frac{1}{\sqrt{n}} \sum_{i=1}^ng_i\hat\psi_j(y_i,z_i)$ and define
 $$
  \widetilde Z^*= \max_{j\in S} |\widehat{\mathcal{G}}_j|, \ \ \bar Z^* = \max_{j\in S} \left|\frac{1}{\sqrt{n}} \sum_{i=1}^ng_i\bar\psi_j(y_i,z_i)\right|, \ \ \mbox{and} \ \
 \widetilde Z  = \max_{j\in S} |\mathcal{G}_j|$$ where $(\mathcal{G}_j)_{j\in S}$ is a zero mean Gaussian vector with covariance matrix given by $C_{jk}=\Ep[\bar\psi_j(y,z)\bar\psi_k(y,z)]$, and $\widetilde Z^*$ is associated with the multiplier bootstrap as defined in (\ref{def:mb}).

 We have that
\begin{equation}\label{aux:step3}\begin{array}{rl}
 |\widetilde Z^* - \widetilde Z | & \leq \left|\widetilde Z^* -  \bar Z^* \right| + \left| \bar Z^* - \widetilde Z \right|  \\
& \leq {\displaystyle \max_{j\in S}}  \left| \frac{1}{\sqrt{n}} \sum_{i=1}^ng_i\{\hat\psi_j(y_i,z_i) - \bar\psi_j(y_i,z_i)\}\right| + \left|\bar Z^* - \widetilde Z\right| \\ \end{array}
 \end{equation}

To control the first term of the RHS in (\ref{aux:step3}) note that conditional on $(y_i,z_i)_{i=1}^n$, $\frac{1}{\sqrt{n}} \sum_{i=1}^ng_i\{\hat\psi_j(y_i,z_i) - \bar\psi_j(y_i,z_i)\}$ is a zero-mean Gaussian random variable with variance $\En[\{\hat\psi_j(y,z)-\bar\psi_j(y,z)\}^2]$. By Step 2's claim (\ref{step2:main}) and (\ref{Aux:L2Pn}), with probability $1-o(1)$ we have uniformly over $j\in S$
$$\begin{array}{rl}\En[\{\hat\psi_j(y,z)-\bar\psi_j(y,z)\}^2]^{1/2} & \leq |\hat\sigma_j^{-1}\hat\Sigma_j^{-1}-\sigma_j^{-1}\Sigma_j^{-1}|\En[\psi_j^2(y,z,\check\beta_j,\hat\eta^j)]^{1/2} \\
& + \sigma_j^{-1}\Sigma_j^{-1}\En[\{\widetilde \psi_j(y,z)-\psi_j(y,z)\}^2]^{1/2}\\
& \lesssim \varepsilon_n \\
\end{array}$$
since $\hat\sigma_j$, $\sigma_j$, $\Sigma_j$, $\hat \Sigma_j$ and $\En[\hat\psi_j^2(y,z)]^{1/2}$ are bounded away from zero with probability $1-o(1)$. Therefore we have with probability $1-o(1)$ that
$$ \begin{array}{rl}\Ep\left[ \max_{j\in S} \left| \frac{1}{\sqrt{n}} \sum_{i=1}^ng_i\{\hat\psi_j(y_i,z_i) - \bar\psi_j(y_i,z_i)\}\right| \ \mid (y_i,z_i)_{i=1}^n \right] \\
\lesssim \varepsilon_n \sqrt{\log 2|S| } =: I_n
\end{array} $$
by Corollary 2.2.8 in \cite{vdV-W}.

Next we proceed to the second term of the RHS in (\ref{aux:step3}). We will apply Theorem 3.2 in \cite{chernozhukov2015noncenteredprocesses} and a conditional version of Strassen's theorem due to \cite{monrad1991nearby}. Let $\bar Z^*=\max_{j\in S} |X_j|$ and $\widetilde Z=\max_{j\in S}|Y_j|$ where $X_{j}=\frac{1}{\sqrt{n}}\sum_{i=1}^ng_i\bar \psi_j(y_i,z_i)$, $j\in S$, and $Y \sim N(0,\Ep[\bar\psi_S(y_1,y_1)\bar\psi^T_S(y_1,z_1)])\in \mathbb{R}^{|S|}$. By construction we have $\Ep[X_j]=\Ep[Y_j]=0$ and $\Ep[X_j^2]=\Ep[Y_j^2]=1$ for all $j\in S$. For a given threshold $\bar\Delta=\bar\Delta_n>0$,  consider the event $E_n=\{ \Delta \leq \bar \Delta \}$ where $$\Delta=\max_{j,k\in S}| \En[\bar\psi_j(y,z)\bar\psi_k(y,z)] - \Ep[\bar\psi_j(y_1,z_1)\bar\psi_k(y_1,z_1)]|.$$
Conditionally on $E_n$,  Theorem 3.2 in \cite{chernozhukov2015noncenteredprocesses} established that for every $\delta>0$ and every Borel subset $A\subset \mathbb{R}$
 $$ P( \max_{j\in S} X_j \in A \mid (y_i,z_i)_{i=1}^n) \leq P(\max_{j\in S} Y_j \in A^\delta ) + C\delta^{-1}\sqrt{\bar \Delta \log(|S|)}$$
for a universal constant $C>0$ where $A^{\delta}=\{ t \in \mathbb{R} : {\rm dist}(t,A)\leq \delta\}$. In turn, by a conditional version of Strassen's theorem, Theorem 4 in \cite{monrad1991nearby} (see also Lemma 4.2 in \cite{chernozhukov2015noncenteredprocesses}), there is a version of $\widetilde Z$  such that \begin{equation}\label{multiplier:coupling}P( | \bar Z^* - \widetilde Z | > \delta ) \leq P(E_n^c) + C\delta^{-1}\sqrt{\bar \Delta \log(|S|)}.\end{equation}
The result will follow from a suitable choice of $\bar \Delta \to 0$ and $\delta \to 0$. Let $m_S:=\max_{j\in S} \sigma_j^{-1}\Sigma_j^{-1}(1+\|\beta_0\|)(1+\|\mu^j_0\|)$. By symmetrization arguments (Lemma 2.3.6 in \cite{van1996weak}), for i.i.d. Rademacher random variables $(r_i)_{i=1}^n$, we have
$$
\begin{array}{rl}
\Ep[\Delta] & \leq C \Ep[  \Ep_{r} [ \En[r\bar \psi_{j}(y,z)\bar\psi_{k}(y,z)] ] \\
& \leq_{(1)} C\sqrt{n^{-1}\log(|S|^2)} \Ep[ \max_{j,k\in S} \En[  \bar \psi_{j}^2(y,z)\bar\psi_{k}^2(y,z)]^{1/2}]\\
& \leq_{(2)} C\sqrt{n^{-1}\log(|S|^2)} \{ \Ep[\max_{i\leq n, j\in S}|\bar \psi_{j}(y_i,z_i)|^4] n^{-1}\log|S|\}^{1/2}\\
 & +  C\sqrt{n^{-1}\log(|S|^2)} \{\max_{j\in S} \Ep[ \bar \psi_{j}^4(y,z)] \}^{1/2}\\
 & \leq_{(3)}  Cn^{-1}\log(|S|)m_S^2 \log^2(n|S|) \\
 & + C\sqrt{n^{-1}\log(|S|^2)} m_S^{3/2}\\
 & \leq_{(4)} C' \sqrt{n^{-1}\log(|S|)} m_S^{3/2}
\end{array}
$$ where (2) follows by Cauchy-–Schwarz inequality and Lemma \ref{lem:m2bound} (part 2), (3) by Lemma \ref{lem:subgaussian} and Lemma \ref{lem:auxmoment} (with $k=4$ and $\delta=1$), and (4) holds by the assumed condition that $n^{-1}m_S^4\log^7(|S|) = o(1)$.

For $\gamma \in (0,1)$, we can set $\bar \Delta = \gamma^{-1}\sqrt{n^{-1}\log(|S|)} m_S^{3/2}$, we have $P(E_n^c)=O(\gamma)$ and by setting $\delta = \gamma^{-1}\{\bar\Delta \log |S|\}^{1/2}$ we have by (\ref{multiplier:coupling}) that
{\small $$ \begin{array}{rl}
\displaystyle |\bar Z^* - \widetilde Z | &\displaystyle  = O_\P\left( n^{-1/4}\log^{3/4}(|S|) m_S^{3/4} \right) =: II_n
\end{array} $$}
where $II_n\to 0$ under $n^{-1}m_S^{4}\log^7(|S|)=o(1)$.

Define $r_n:= \ell_n(I_n + II_n)\to 0$ for some $\ell_n\to \infty$, and $\wp_n = P( |\widetilde Z^* - \widetilde Z| > r_n )^{1/2} =o(1)$.
Letting $U_n := P( |\widetilde Z^* - \widetilde Z| > r_n \mid (y_i,z_i)_{i=1}^n)$ note that $$\begin{array}{rl}
P(  U_n > \wp_n ) & = \Ep [ 1\{ U_n > \wp_n \} ]\\
& \leq  \Ep [ U_n ] / \wp_n \\
& = \wp_n\\
\end{array}
$$ so that $P( |\widetilde Z^* - \widetilde Z| > r_n \mid (y_i,z_i)_{i=1}^n) \leq \wp_n$ with probability at least $1-\wp_n$. Then, by definition of the quantile function we have that
$$\begin{array}{rl}
(1+\varepsilon_n)c_\alpha^* & \leq (1+\varepsilon_n)(c^0_{\alpha - \wp_n}+r_n) \\
& = c^0_{\alpha-\vartheta_n}- \{ c^0_{\alpha-\vartheta_n} - c^0_{\alpha - \wp_n}\} +c^0_{\alpha - \wp_n} \varepsilon_n + (1+\varepsilon_n)r_n\\
& \leq c^0_{\alpha-\vartheta_n} - \frac{c(\vartheta_n- \wp_n)}{\Ep[\widetilde Z]} +  c^0_{\alpha - \wp_n}\varepsilon_n + (1+\varepsilon_n)r_n\\
& \leq c^0_{\alpha-\vartheta_n}
\end{array}
$$ where the third step we used Corollary 2.1 in \cite{chernozhukov2014honestbands}, and we set $\vartheta_n\to 0$ so that $\Ep[\widetilde Z] \{  c^0_{\alpha - \wp_n}\varepsilon_n + r_n(1+\varepsilon_n) \} = o(\vartheta_n- \wp_n)$. This is possible since  $\Ep[\widetilde Z] \leq C\sqrt{1 + \log |S|}$, $ c^0_{\alpha - \wp_n} \leq C\sqrt{1 + \log (|S|/\{\alpha - \wp_n\})}$, so that
$$ \varepsilon_n \log |S| \leq C \delta_n \log^{3/2}(|S|)(1+\|\beta_0\|)(1+\max_{j\in S}\|\mu^j_0\|)=o(1)$$ under the assumed condition (i) in the statement of the theorem, and
 $$\begin{array}{rl}
 r_n \sqrt{\log |S|} & = \ell_n(I_n+II_n) \\
 & = \ell_n O( \varepsilon_n\sqrt{\log|S|} + n^{-1/4}m_S^{3/4}\log^{3/4}|S|)\\
 & = \ell_n o(1) + \ell_n O(n^{-1/4}m_S^{3/4}\log^{3/4}|S|) = o(1)\end{array}$$
by choosing $\ell_n \to \infty$ slowly enough and  $n^{-1/4}m_S^{3/4}\log^{3/4}|S|=o(1)$ by condition (ii) in the statement of the theorem.

\end{proof}

\begin{proof}[Proof of Theorem \ref{thm:mainEst}]

For $\hat\eta^j = [ \hat\beta ; \hat\mu^j]$ we can rewrite $\psi_j$ as
$$ \psi_j(y,z,\theta,\hat\eta^j)  = (z_j-z_{-j}^T\hat\mu^j)(y-z_j\theta-z^T_{-j}\hat\beta_{-j})+\hat \Gamma_{jj}\theta -(\hat\mu^j)^T\hat \Gamma_{-j,-j}\hat\beta_{-j}$$
Then, we can achieve $0 = \frac{1}{n}\sum_{i=1}^n\psi_j(y_i,z_i,\theta,\hat\eta^j)$ by setting $\theta = \check\beta_j$ defined as
\begin{equation}\label{def:newbetaj} \check\beta_j := \frac{\hat \Sigma_j^{-1}}{n}\sum_{i=1}^n(z_{ij}-z_{i,-j}^T\hat\mu^j)(y_i-z_i^T\hat\beta_{-j})-(\hat \mu^j)^T\hat \Gamma_{-j,-j}\hat\beta_{-j}\end{equation}
where $\hat \Sigma_j := \left\{\frac{1}{n}\sum_{i=1}^n (z_{ij}-z_{i,-j}^T\hat\mu^j)z_j - \hat \Gamma_{jj} \right\}$.
Next we rearrange the expression (\ref{def:newbetaj}). We will use the notation
$\En[\cdot]=\frac{1}{n}\sum_{i=1}^n[\cdot_i]$. It follows that

{\small $$\begin{array}{rl}
  \check\beta_j & = \hat \Sigma_j^{-1}\En[(z_{j}-z_{-j}^T\hat\mu^j)(y-z^T_{-j}\hat\beta_{-j})-(\hat \mu^j)^T\hat \Gamma_{-j,-j}\hat\beta_{-j}]\\
  & = \hat \Sigma_j^{-1}\{\En[(z_{j}-z_{-j}^T\hat\mu^j)z_j]\beta_{0j} -\hat\Gamma_{jj}\beta_{0j} \\
  & +  \En[(z_{j}-z_{-j}^T\hat\mu^j)(y-z_j\beta_{0j}-z^T_{-j}\hat\beta_{-j})+ \hat\Gamma_{jj}\beta_{0j}-(\hat \mu^j)^T\hat \Gamma_{-j,-j}\hat\beta_{-j}]\}\\
  & =\beta_{0j}\\
  &+\hat \Sigma_j^{-1}\En[(z_{j}-z_{-j}^T\hat\mu^j)(y-z_j\beta_{0j}-z^T_{-j}\hat\beta_{-j})+ \hat\Gamma_{jj}\beta_{0j}-(\hat \mu^j)^T\hat \Gamma_{-j,-j}\hat\beta_{-j}]\\
\end{array}$$}
In turn we can rewrite
$$
\begin{array}{rl}
\En[(z_{j}-z_{-j}^T\hat\mu^j)(y-z_j\beta_{0j}-z^T_{-j}\hat\beta_{-j})+ \hat\Gamma_{jj}\beta_{0j}-(\hat \mu^j)^T\hat \Gamma_{-j,-j}\hat\beta_{-j}]\\
= \En[(z_{j}-z_{-j}^T\mu^j_0)(y-z_j\beta_{0j}-z^T_{-j}\beta_{0,-j})+ \Gamma_{jj}\beta_{0j}-(\mu^j_0)^T \Gamma_{-j,-j}\beta_{0,-j}] \\
 + (\hat\Gamma_{jj}-\Gamma_{jj})\beta_{0j} -(\mu^j_0)^T(\hat \Gamma_{-j,-j}-\Gamma_{-j,-j})\beta_{0,-j} \\
 + T_1 + T_2 + T_3 + T_4 + T_5
\end{array}
$$ where
$$\begin{array}{rl}
T_1 & = \En[(z_{j}- z_{-j}^T\mu^j_0)z^T_{-j}+(\mu^j_0)^T \Gamma_{-j,-j}](\beta_{0,-j}-\hat\beta_{-j})\\
T_2 & =(\mu^j_0)^T(\hat \Gamma_{-j,-j}-\Gamma_{-j,-j})(\beta_{0,-j}-\hat\beta_{-j}) \\
T_3 & =  (\mu^j_0-\hat\mu^j)^T\En[z_{-j}(y-z^T\beta_{0})+\Gamma_{-j,-j}\beta_{0,-j}]\\
T_4 & =(\mu^j_0-\hat\mu^j)^T\En[z_{-j}z^T -\hat \Gamma_{-j,-j}](\beta_{0,-j}-\hat\beta_{-j})\\
T_5 & = (\mu^j_0-\hat\mu^j)^T(\hat \Gamma_{-j,-j}-\Gamma_{-j,-j})\beta_{0,-j}
\end{array}
$$

By Step 2 below the quantities $T_k, k=1,\ldots,5$ satisfy with probability $1-o(1)$
$$|T_1 + T_2 + T_3 + T_4 + T_5 | \leq C(1+\max_{j\in S}\|\mu_0^j\|)(1+\|\beta_0\|) s \log (pn) / n.  $$

Under Condition C, we have that $(1+\max_{j\in S}\|\mu_0^j\|)(1+\|\beta_0\|) s \log (pn) \leq \delta_n \sqrt{n}$, so that we obtain the following linear representation for the estimator uniformly over $j\in S$
\begin{equation}\label{middlemainstep}\begin{array}{rl}
\sqrt{n}\hat \Sigma_j(\check\beta_j-\beta_{0j}) & = \frac{1}{\sqrt{n}}\sum_{i=1}^n\psi_j(y_i,z_i) + \sqrt{n}(\hat\Gamma_{jj}-\Gamma_{jj})\beta_{0j} \\
&  -(\mu^j_0)^T\sqrt{n}(\hat \Gamma_{-j,-j}-\Gamma_{-j,-j})\beta_{0,-j} + O_\P(\delta_n)\end{array}  \end{equation}
where $\psi_j(y_i,z_i)=(z_{j}- z_{-j}^T\mu^j_0)(y-z^T\beta_{0})+ (e^j-\mu^j_0)^T\Gamma\beta_{0}$ is a zero mean random variable. Note that this is equivalent to
{\small $$
\sqrt{n}\hat \Sigma_j(\check\beta_j-\beta_{0j}) = \frac{1}{\sqrt{n}}\sum_{i=1}^n(z_{ij}-z_{i,-j}^T\mu^j_0)(y_i-z_i^T\beta_{0})+ (e^j-\mu^j_0)^T\hat \Gamma\beta_{0}  + O_\P(\delta_n) \\
$$}

The result follows provided we show
$$ \max_{j\in S} \left|\sqrt{n}(\hat \Sigma_j-\Sigma_j)(\check\beta_j-\beta_{0j})\right| = O_\P(\delta_n) $$

It follows that with probability $1-o(1)$, uniformly over $j\in S$
\begin{equation}\label{bound:hatSigmaj}\begin{array}{rl}
|\hat \Sigma_j-\Sigma_j| & \leq | \En[z_{-j}^T\{\hat \mu_0^j-\mu^j_0\}z_j]| + | \hat\Gamma_{jj}- \Gamma_{jj} |\\
& + |\En[(z_j-z_{-j}^T \mu_0^j)z_j]-\Ep[(z_j-z_{-j}^T \mu_0^j)z_j] |\\
& \leq_{(1)} | \En[(z_{-j}^T\{\hat \mu_0^j-\mu^j_0\})^2]^{1/2}\En[z_j^2]^{1/2} + \| \hat\Gamma- \Gamma \|_\infty \\
& +  |\En[(z_j-z_{-j}^T \mu_0^j)z_j]-\Ep[x_j^2+\Gamma_{jj}] |\\
& \leq_{(2)} | \En[(z_{-j}^T\{\hat \mu_0^j-\mu^j_0\})^2]^{1/2}\En[z_j^2]^{1/2} + C\sqrt{\log(pn)/n}\\
& \leq_{(3)} C(1+\max_{j\in S}\|\mu^j_0\|)\sqrt{s\log(pn)/n}\\
& \leq_{(4)} C\delta_n/\{(1+\|\beta_0\|)\sqrt{s\log(pn)}\}\\
\end{array}\end{equation}
where (1) follows by Cauchy--Schwarz, (2) follows by Condition D, Lemma \ref{lem:extra} with $\log(pn)\leq \delta_n n$ implied by Condition A, (3) Condition B and bounded $Cs$-sparse eigenvalues of the matrix $\En[zz^T]$ with probability $1-o(1)$, and (4) by Condition C.

Next note that by (\ref{middlemainstep}), we have
$$\begin{array}{rl}
\sqrt{n}\|\check\beta_S - \beta_{0S}\|_\infty & \leq \max_{j\in S} \hat\Sigma_j^{-1} \sigma_j\Sigma_j \max_{j\in S} \left|\frac{1}{\sqrt{n}} \sum_{i=1}^n\bar\psi_j(y_i,z_i)\right|\\
& +  \max_{j\in S} \hat\Sigma_j^{-1}\sqrt{n}\|\hat\Gamma-\Gamma\|_\infty\|\beta_0\|_\infty \\
& +  \max_{j\in S} \hat\Sigma_j^{-1}\sqrt{n}\|\mu^j_0\| \ \|\hat\Gamma-\Gamma\|_{op}\|\beta_0\| \\
& +  \max_{j\in S} \hat\Sigma_j^{-1} O_\P(\delta_n)\\
& \leq \max_{j\in S} \hat\Sigma_j^{-1} \sigma_j\Sigma_j \max_{j\in S} \left|\frac{1}{\sqrt{n}} \sum_{i=1}^n\bar\psi_j(y_i,z_i)\right|\\
& + C(\max_{j\in S} \hat\Sigma_j^{-1}) \{\sqrt{\log(pn)}(1+\|\beta_0\|)+ O_\P(\delta_n)\}\\
\end{array}$$
since $\|\hat\Gamma-\Gamma\|_{op}=\|\hat\Gamma-\Gamma\|_\infty \leq C\sqrt{\log(pn)/n}$ by Condition D.
Note that by inspection of the proof of Corollary \ref{thm:Linfty}, relation (\ref{bound:sum_psi_bar}), we have that with probability $1-o(1)$
$$  \max_{j\in S} \left|\frac{1}{\sqrt{n}} \sum_{i=1}^n\bar\psi_j(y_i,z_i)\right| \leq C\sqrt{\log (pn)}.$$
Combining these relations we have with probability $1-o(1)$ that
$$ \begin{array}{rl}
& \max_{j\in S} \left|\sqrt{n}(\hat \Sigma_j-\Sigma_j)(\check\beta_j-\beta_{0j})\right| \\
& \leq \frac{C\delta_n(1+\|\beta_0\|)^{-1}}{\sqrt{s\log(pn)}} \left\{ {\displaystyle \max_{j\in S} } \frac{\sigma_j\Sigma_j}{\hat\Sigma_j} \sqrt{\log(pn)} + \sqrt{\log(pn)}(1+\|\beta_0\|)+ O_\P(\delta_n) \right\}\\
& = O_\P(\delta_n) \end{array}$$
as needed.

Step 2. (Auxiliary Calculations for $T_k, k=1,\ldots,5$.) We note that the bounds will hold over uniformly over $j\in S$. We start with $T_5$. It follows that with probability $1-o(1)$ that
$$\begin{array}{rl}
|T_5| & = |(\mu^j_0-\hat\mu^j)^T(\hat \Gamma_{-j,-j}-\Gamma_{-j,-j})\beta_{0,-j}|\\
&  \leq \|\mu^j_0-\hat\mu^j\| \ \ \|\hat \Gamma_{-j,-j}-\Gamma_{-j,-j}\|_{op} \ \ \|\beta_0\|\\
& \leq C(1+\|\mu^j_0\|)\|\beta_0\|\sqrt{s \log(p)/n} \|\hat \Gamma_{-j,-j}-\Gamma_{-j,-j}\|_{\infty}
\end{array}
$$ by the $\ell_2$-rate of convergence of $\hat\mu^j$ assumed on Condition B, and the matrices $\hat\Gamma$ and $\Gamma$ are diagonal by Condition D.

The bound of $T_4$ note that under Condition A, the matrix $\En[z_{-j}z^T -\hat \Gamma_{-j,-j}]$ has sparse eigenvalues of order $2(C+1)s$ bounded from above with probability $1-o(1)$. Indeed by Lemma \ref{thm:KLthm4} with $t := \log n + 2\log \binom{p}{2C+1}$, 
a the relation $s^2\log^2(pn) \leq \delta_n n$ implied by Condition C, and $\|\hat \Gamma\|_{op} \leq C$. Furthermore, by Condition B we have $\|\beta_0\|_0\leq s$ and $\|\mu^j_0\|_0\leq s$ and with probability $1-o(1)$ we have $\|\hat\beta\|_0\leq Cs$ and $\|\hat\mu^j\|_0 \leq Cs$.  Therefore we with probability $1-o(1)$ that
$$\begin{array}{rl}
|T_4| & =|(\mu^j_0-\hat\mu^j)^T\En[z_{-j}z^T -\hat \Gamma_{-j,-j}](\beta_{0,-j}-\hat\beta_{-j})|\\
& \leq \|\mu^j_0-\hat\mu^j\| \ \sup_{\|\delta\|_0\leq 2(C+1)s, \|\delta\|=1} \delta^T\En[zz^T -\hat \Gamma]\delta  \ \  \|\beta_{0,-j}-\hat\beta_{-j}\|\\
& \leq C (1+\|\mu_0^j\|)(1+\|\beta_0\|) s \log (pn) / n\\
\end{array}
 $$ where the last bound holds  by Condition B to bound the $\ell_2$-rates $\|\mu^j_0-\hat\mu^j\|$ and $\|\beta_{0,-j}-\hat\beta_{-j}\|$.

In order to control $T_3$ note that, under Condition A, with probability $1-\varepsilon -o(1)$
$$\|\En[z(y-z^T\beta_{0})+\Gamma\beta_{0}]\|_\infty \leq C (1+\|\beta_0\|)\sqrt{\log(p/\varepsilon)/n}$$ so that
with probability $1-\varepsilon - o(1)$ we have
$$\begin{array}{rl}
|T_3| & = |(\mu^j_0-\hat\mu^j)^T\En[z_{-j}(y-z^T\beta_{0})+\Gamma_{-j,-j}\beta_{0,-j}]| \\
& \leq \|\mu^j_0-\hat\mu^j\|_1\|\En[z_{-j}(y-z^T\beta_{0})+\Gamma_{-j,-j}\beta_{0,-j}]\|_\infty\\
& \leq C(1+\|\mu^j_0\|)(1+\|\beta_0\|)s\log(p/\varepsilon)/n
\end{array}
$$
where we used Condition B to bound the $\ell_1$-rate of convergence of $\hat\mu^j$.

The bound on $T_2$ follows from
$$\begin{array}{rl}
|T_2| & = |(\mu^j_0)^T(\hat \Gamma_{-j,-j}-\Gamma_{-j,-j})(\beta_{0,-j}-\hat\beta_{-j})| \\
& \leq \|\mu_0^j\|\|\hat \Gamma_{-j,-j}-\Gamma_{-j,-j}\|_{op}\|\beta_{0,-j}-\hat\beta_{-j}\|\\
& = \|\mu_0^j\|\|\hat \Gamma_{-j,-j}-\Gamma_{-j,-j}\|_\infty\|\beta_{0,-j}-\hat\beta_{-j}\|\\
& \leq \|\mu_0^j\|(1+\|\beta_0\|)\sqrt{s\log(p)/n} \|\hat \Gamma_{-j,-j}-\Gamma_{-j,-j}\|_\infty\\
\end{array}
$$ where we used that the matrices $\hat \Gamma$ and $\Gamma$ are diagonal and Condition B.

To control $T_1$ we have
$$\begin{array}{rl}
|T_1| & = |\En[(z_{j}-z_{-j}^T\mu^j_0)z^T_{-j}+(\mu^j_0)^T \Gamma_{-j,-j}](\beta_{0,-j}-\hat\beta_{-j})|\\
& \leq \{\|\En[w_{j}z^T_{-j}]\|_\infty+ \|\En[(x_{j}-(\mu^j_0)^T x_{-j})x^T_{-j}]\|_\infty\\
 & +\|\En[(x_{j}- x_{-j}^T\mu^j_0)w^T_{-j}]\|_\infty+ \|\En[(\mu^j_0)^Tw_{-j}x^T_{-j}]\|_\infty \\
 & + \|(\mu^j_0)^T\En[w_{-j}w^T_{-j}-\Gamma_{-j,-j}]\|_\infty\}\|\beta_{0,-j}-\hat\beta_{-j}\|_1\\
\end{array}
$$
We proceed to bound the five terms in the curly bracket. For the first term, under Condition A, by Lemma \ref{lem2}  we have with probability $1-\varepsilon-o(1)$
$$
\begin{array}{rl}
\|\En[w_{j}z^T_{-j}]\|_\infty & \leq \|\En[w_{j}x^T_{-j}]\|_\infty+\|\En[w_{j}w^T_{-j}]\|_\infty\\
 &\leq \|\mbox{$\frac{1}{n}$}X^TW\|_\infty + \|\mbox{$\frac{1}{n}$}(W^TW-{\rm diag}(W^TW))\|_\infty \\
 & \leq C\sqrt{\frac{\log(p^2/\varepsilon)}{n}}
\end{array}$$

To bound the second term we use that $\Ep[(x_{j}- x_{-j}^T\mu^j_0)x^T_{-j}]=0$. Thus by Lemma \ref{lem:extra2} with $\theta^j=e^j-\mu^j_0$ we have
$$
\begin{array}{rl}
\|\En[(x_{j}- x_{-j}^T\mu^j_0)x^T_{-j}]\|_\infty & = \max_{j\leq p, k\neq j} |\En[(x_{j}- x_{-j}^T\mu^j_0)x_{k}]|\\
& \leq  C\sqrt{\log (p^2/\varepsilon)/n}
\end{array}$$ with probability $1-\varepsilon$ under $\log(pn) \leq \delta_n\sqrt{n}$ for $\varepsilon=n^{-1}$.

The third term can be controlled by a variant of Lemma \ref{lem2} (applied with $x_{j}-x_{-j}^T\mu^j_0$ instead of $x_j$, since $\|\mu^j_0\|\leq C$ and $\Ep[\{x_{j}-x_{-j}^T\mu^j_0\}^2]\leq \Ep[x_j^2]$), namely with probability $1-\varepsilon$
$$ \begin{array}{rl}
\|\En[(x_{j}-x_{-j}^T\mu^j_0)w^T_{-j}]\|_\infty & \leq \|\mbox{$\frac{1}{n}$}W^T(X_j-X_{-j}\mu^j_0)\|_\infty \\
 & \leq \max_{1\leq j\leq p}\|\mbox{$\frac{1}{n}$}(X_j-X_{-j}\mu^j_0)^TW\|_\infty \\
 & \leq C\sqrt{\frac{\log(p^3/\varepsilon)}{n}}
\end{array}$$
under Conditions A and C. We will take $\varepsilon = n^{-1}$.

Regarding the fourth term we have that by Lemma \ref{lem3a}, with probability $1-\varepsilon$
$$ \begin{array}{rl}
\|\En[(\mu^j_0)^Tw_{-j}x^T_{-j}]\|_\infty & \leq \max_{1\leq j,k\leq p} |\En[(\mu^j_0)^Tw_{-j}x_{k}]| \\
& \leq C \max_{j\in S}\|\mu^j_0\|\sigma_w\sqrt{\log(2p^2/\varepsilon)/n} \\
\end{array}$$
under Conditions A and C. We will take $\varepsilon = n^{-1}$.

Finally, the last term satisfies with probability $1-2\varepsilon$
$$ \begin{array}{rl}
\|(\mu^j_0)^T\En[w_{-j}w^T_{-j}-\Gamma_{-j,-j}]\|_\infty & \leq \|(\mu^j_0)^T\En[w_{-j}w^T_{-j}-{\rm Diag}(w_{-j}w^T_{-j})]\|_\infty\\
& +\|(\mu^j_0)^T\En[{\rm Diag}(w_{-j}w^T_{-j})-\Gamma_{-j,-j}]\|_\infty\\
& \leq C\max_{j\leq p}\|\mu^j_0\|\left\{ \delta''(\varepsilon/p) + \delta(\varepsilon)\right\}
\end{array}$$
where in the last step we used Lemma \ref{lem3a} and Lemma \ref{lem2}. We will take $\varepsilon = n^{-1}$.
\end{proof}

\section{Auxiliary Lemmas}

The following technical lemmas are modifications of results in \cite{BRT2014} and \cite{RT2} to allow random design and are stated here for completeness. In what follows,  for a square matrix $A$, we denote
by $\text{Diag}\{A\}$ the matrix with the same dimensions as $A$,
the same diagonal elements, and all off-diagonal elements
equal to zero.

\begin{lemma}\label{lem2} Let $0<\e<1$, $p\geq 2$, and assume Condition A holds. Then, with probability at least $1-\varepsilon$ (for each event),
\begin{eqnarray*}
&&\left\|\mbox{$\frac{1}{n}$} X^TW\right\|_\infty \le \delta(\varepsilon,p),\quad \left\| \mbox{$\frac{1}{n}$} X^T\xi\right\|_\infty \le \delta(\varepsilon,p),\quad \left\|\mbox{$\frac{1}{n}$} W^T\xi\right\|_\infty \le \delta(\varepsilon,p),\\
&& \left\|\mbox{$\frac{1}{n}$}(W^TW-\text{\rm Diag}\{W^TW\})\right\|_\infty \le \delta(\varepsilon,p),\\
&&\left\|\mbox{$\frac{1}{n}$} \text{\rm Diag}\{W^TW\}-\Gamma\right\|_\infty \le \delta(\varepsilon,p),\\
&& \left\|\mbox{$\frac{1}{n}$} \text{\rm Diag}\{X^TX\}-\Ep[\mbox{$\frac{1}{n}$} \text{\rm Diag}\{X^TX\}]\right\|_\infty \le \delta(\varepsilon,p)
\end{eqnarray*}
for
$$
\delta (\varepsilon,p) = \max\left(\gamma_0
\sqrt{\frac{2\log(2p^2/\varepsilon)}{n}},\
\frac{2\log(2p^2/\varepsilon)}{t_0n}\right),
$$
where $\gamma_0, t_0$ are positive constants depending only on $\sigma_\xi, \sigma_w, \sigma_x$.
\end{lemma}

\begin{lemma} \label{lem3a}
Let $0<\e<1$, $\theta^*\in \R^p$, $p\geq 2$, and assume that Condition A holds. Then, with probability at least $1-\varepsilon$,
\begin{align*}
&\left\|\mbox{$\frac{1}{n}$} X^TW\theta^*\right\|_{\infty} \leq \delta'(\e,p)\|\theta^*\|_2,
\end{align*}
where $\delta'(\e,p)= \max\left(\gamma_1
\sqrt{\frac{2\log(2p^2/\varepsilon)}{n}},\
\frac{2\log(2p^2/\varepsilon)}{t_1n}\right),
$
where $\gamma_1, t_1$ are positive constants depending only on $\sigma_w, \sigma_x$.
In addition, 
with probability at least $1-\varepsilon$,
\begin{align*}
&\left\|\mbox{$\frac{1}{n}$} (W^TW-{\rm Diag}\{W^TW\})\theta^*\right\|_{\infty}\leq\delta''(\e,p)\|\theta^*\|_2,
\end{align*}
where 
$
\delta''(\e,p) =\max\left(\gamma_2
\sqrt{\frac{2\log(2p/\varepsilon)}{n}},\
\frac{2\log(2p/\varepsilon)}{t_2n}\right),
$
and $\gamma_2, t_2$ are positive constants depending only on $\sigma_w$.
\end{lemma}

\begin{lemma}\label{lem:extra}
Let $0<\e<1$, $\theta^*\in \R^p$, $S\subset\{1,\ldots,p\}$, $|S|\geq 2$, and assume that Condition A holds. Then, with probability at least $1-\varepsilon$,
$$ \max_{j\in S} \left|\frac{1}{n}\sum_{i=1}^n (z_{ij}-z_{i,-j}^T\mu^j_0)z_{ij} - \Ep[(z_{ij}-z_{i,-j}^T\mu^j_0)z_{ij}] \right| \leq 4\delta(\e,|S|)\{1+\max_{j\in S}\|\mu^j_0\|\}$$
for
$$
\delta (\varepsilon,|S|) = \max\left(\gamma_3
\sqrt{\frac{2\log(2|S|^2/\varepsilon)}{n}},\
\frac{2\log(2|S|^2/\varepsilon)}{t_3n}\right),
$$
where $\gamma_3, t_3$ are positive constants depending only on $\sigma_w, \sigma_x$.
\end{lemma}
\begin{proof}
Note that $(z_{ij}-z_{i,-j}^T\mu^j_0)z_{ij} = x_{ij}^2 + w_{ij}^2 +2x_{ij}w_{ij}-x_{ij}x_{i,-j}^T\mu^j_0-w_{-j}^T\mu^j_0x_j-w_jx_{-j}^T\mu^j_0-w_{-j}^T\mu^j_0w_j$ so that
$\Ep[(z_{ij}-z_{i,-j}^T\mu^j_0)z_{ij}]=\Ep[x_{ij}^2 + w_{ij}^2-x_{ij}x_{i,-j}^T\mu^j_0$.
Then using Lemma \ref{lem2} we have that
$$ \max_{j\in S} \left|\frac{1}{n}\sum_{i=1}^n (z_{ij}-z_{i,-j}^T\mu^j_0)z_{ij} - \Ep[(z_{ij}-z_{i,-j}^T\mu^j_0)z_{ij}] \right| \leq \sum_{k=1}^7 r_k
$$
where
$$
\begin{array}{ll}
 r_1 := \max_{j\in S}| \frac{1}{n}\sum_{i=1}^n x_{ij}^2-\Ep[x_{ij}^2]|  &  r_2 :=  \max_{j\in S}| \frac{1}{n}\sum_{i=1}^n w_{ij}^2-\Ep[w_{ij}^2]| \\
 r_3 :=  2\max_{j\in S}| \frac{1}{n}\sum_{i=1}^n w_{ij}x_{ij}|  &   r_4 := \max_{j\in S}| \frac{1}{n}\sum_{i=1}^n w_{ij}w_{i,-j}^T\mu^j_0| \\
   r_5:= \max_{j\in S}| \frac{1}{n}\sum_{i=1}^n x_{ij}w_{i,-j}^T\mu^j_0| &  r_6 := \max_{j\in S}| \frac{1}{n}\sum_{i=1}^n w_{ij}x_{i,-j}^T\mu^j_0| \\
\multicolumn{2}{l}{r_7 := \max_{j\in S}| \frac{1}{n}\sum_{i=1}^n x_{ij}x_{i,-j}^T\mu^j_0-\Ep[x_{ij}x_{i,-j}^T\mu^j_0]|} \\
  \end{array}
$$ By Lemma \ref{lem2} with $S$ (instead of all $p$ components) we have that with probability $1- 3\varepsilon$
$$ r_1 \leq \delta(\varepsilon,|S|), r_2 \leq \delta(\varepsilon,|S|), \ \ \mbox{and} \ \ r_3 \leq 2 \delta(\varepsilon,|S|).$$
Similarly, by Lemma \ref{lem3a} we have with probability $1-4\varepsilon$ (noting that each $j$ component is zero mean)
$$ \begin{array}{rl}
r_4 \leq \delta(\varepsilon,|S|)\max_{j\in S}\|\mu^j_0\|, & r_5 \leq \delta(\varepsilon,|S|)\max_{j\in S}\|\mu^j_0\|, \\
r_6 \leq \delta(\varepsilon,|S|)\max_{j\in S}\|\mu^j_0\|, & r_7 \leq \delta(\varepsilon,|S|)\max_{j\in S}\|\mu^j_0\|.\end{array}$$
Combining these bounds yields the result.
\end{proof}

\begin{lemma}\label{lem:extra2}
Under Condition A, let $0<\e<1$, $S\subset\{1,\ldots,p\}$, $|S|\geq 2$, and $\theta^j\in \R^p$, such that $\Ep[x^T\theta^jx_{-j}]=0$. Then, with probability at least $1-\varepsilon$,
$$ \max_{j\in S} \left\|\frac{1}{n}\sum_{i=1}^n x_i^T\theta^jx_{i,-j} \right\|_\infty \leq \delta(\e,p|S|)\max_{j\in S}\|\theta^j\|$$
for
$$
\delta (\varepsilon,p|S|) = \max\left(\gamma_4
\sqrt{\frac{2\log(2p|S|/\varepsilon)}{n}},\
\frac{2\log(2p|S|/\varepsilon)}{t_4n}\right),
$$
where $\gamma_4, t_4$ are positive constants depending only on $\sigma_x$.
\end{lemma}
\begin{proof}
For each $j\in S$ and $k \in [p]\setminus\{j\}$ we proceed to bound $|\En[x^T\theta^jx_{k}]$ and then apply the union bound.  By Condition A we have that $x^T\theta^j$ is subgaussian with variance parameter bounded by $\sigma_x^2\|\theta^j\|^2$ and $x_{ik}$ is subgaussian with variance parameter $\sigma_x^2$. Therefore we have that $(x_i^T\theta^jx_{ik})_{i=1}^n$ are independent zero mean subexponential random variable with parameter $\|\theta^j\|\sigma_x^2$. By Proposition 5.16 in \cite{vershynin2010introduction} we have $$P\left( \left|\sum_{i=1}^nx_i^T\theta^jx_{ik}\right|\geq t \right)\leq 2\exp(-c\min(t^2/\{n\|\theta^j\|^2\sigma_x^4\}, t/\{\|\theta^j\|\sigma_x^2\})) $$
So that setting $t=\max\left\{\sqrt{n\log(2|S|p/\e)}, \log(2|S|p/\e)\right\} \max_{j\in S}\|\theta^j\|\sigma_x^2(1+1/c)$ and applying the union bound yields the result.
\end{proof}

The following technical lemma is a concentration bound, see \cite{BRT2014} for a proof.

\begin{lemma}\label{lem:m2bound}
Let $X_i, i=1,\ldots,n,$ be independent random vectors in $\mathbb{R}^p$, $p\geq 3$. Define $\bar m_k := \max_{j\leq p}\frac{1}{n}\sum_{i=1}^n\mathbb{E}[|X_{ij}|^k]$ and $M_{k} \geq \mathbb{E}[ {\displaystyle \max_{i\leq n}}\|X_i\|_\infty^k]$. Then
{\small $$\mathbb{E}\left[\max_{j\leq p}\frac{1}{n}\left|\sum_{i=1}^n|X_{ij}|^k-\mathbb{E}[|X_{ij}|^k]\right|\right] \leq 2C^2 \frac{\log p}{n}M_{k}+2C\sqrt{\frac{\log p}{n}}M_{k}^{1/2}\bar m_k^{1/2} $$}
$$\mathbb{E}\left[\max_{j\leq p}\frac{1}{n}\sum_{i=1}^n|X_{ij}|^k\right] \leq CM_{k} n^{-1}\log p+ C\bar m_k $$
for some universal constant $C$.
\end{lemma}


The following is a direct consequence of Theorem 2 (and Remark 1) in \cite{koltchinskii2014asymptotics} and the union bound.

\begin{lemma}\label{thm:KLthm4}
Let $(X_{i})$, $i=1,\ldots, n$, be i.i.d. subgaussian random vectors such that $X_{i} \in \mathbb{R}^p$. For $S\subseteq \{1,\ldots,p\}$ let $\Sigma_S=\Ep[X_SX^T_S]$ and ${\rm r}(S) = \tr(\Sigma_S)/\|\Sigma_S\|_{op}$.
Then, there is a universal constant $C$ such that with probability $1-e^{-t}\binom{p}{k}$
{\small $$
\begin{array}{l}
 {\displaystyle \sup_{|S|\leq k, \|\theta\| =1}} \left| \frac{1}{n}\sum_{i=1}^n (\theta_S^TX_i)^2 - \theta_S^T\Sigma\theta_S \right|\\
 \leq C \max_{|S|\leq k}\left\{ \sqrt{\frac{{\rm r}(S)}{n}}\vee \frac{{\rm r}(S)}{n} \vee \sqrt{\frac{t}{n}}\vee \frac{t}{n}\right\} {\displaystyle\sup_{|S|\leq k, \|\theta\| =1} } \sqrt{\theta_S^T\Sigma\theta_S}
\end{array}
$$}
\end{lemma}

Next we collect well-known results of subgaussian random variables that are stated here for convenience.

\begin{lemma}[Technical Lemma for Subgaussian Random Variables]\label{lem:subgaussian}
(1) If $X$ is a centered subgaussian random variable with parameter $\gamma$, it follows that for any $k>0$
$$ \Ep[ |X|^k] \leq k2^{k/2}\Gamma( k/2 )\gamma^k$$
and for $k\geq 1$ we have $\{\Ep[ |X|^k]\}^{1/k} \leq C\gamma \sqrt{k}$ for some universal constant $C$.
(2) If $X_j, j=1,\ldots,N$, is a collection of centered subgaussian variables with parameter $\gamma$, then for $k\geq 1$ we have
$$ \Ep\left[ \max_{j\leq N}|X_j|^k \right] \leq \gamma^k 3^k\log^{k/2}( NC_k ) $$ for some constant $C_k$ that depends only on $k$.\\
\end{lemma}
\begin{proof}
The first two results are standard characterizations of subgaussian random variables. For completeness we show part (2). Recall that if $X_j$ is $\gamma$-subgaussian then $\Ep[\exp(|X_j|^2/\{3\gamma^2\})] \leq 2$ and $\Ep[ |X|^k] \leq k 2^{k/2}\gamma^k\Gamma( k/2 )$ for all $j\leq N$. For $t\geq 0$, define $\psi(t) = \max\{ A_k, \exp(t^{2/k}) \}+t-A_k$ for $A_k=\exp(\frac{1}{2}k-1)$ so that $\psi$ is convex, $\psi(0)=0$, non-negative, and strictly increasing. In particular, $\psi(t)\leq \exp(t^{2/k}) + t$. Therefore we have
$$
\begin{array}{rl}
\psi(\Ep[ \max_{j\leq N}|X_j|^k/\{3\gamma\}^k ]) & \leq \Ep[ \psi(\max_{j\leq N}|X_j|^k/\{3\gamma\}^k) ]\\
  &= \Ep[ \max_{j\leq N}\psi(|X_j|^k/\{3\gamma\}^k) ] \\
& \leq N \max_{j\leq N}\Ep[\psi(|X_j|^k/\{3\gamma\}^k) ] \\
& \leq N \max_{j\leq N}\Ep[\exp(|X_j|^2/\{3\gamma\}^2)]\\
& +N \max_{j\leq N}\Ep[|X_j|^k/\{3\gamma\}^k] \\
& \leq N \{ 2 + k 2^{k/2}\Gamma( k/2 )/3^k\} =: N B_k\\
\end{array}
$$
Thus we have $\Ep[ \max_{j\leq N}|X_j|^k/\{3\gamma\}^k ] \leq \psi^{-1}(N B_k)$. To bound the inverse function note that
$$ NB_k = \psi(\psi^{-1}(N B_k)) = \max\{ A_k, \exp( \{ \psi^{-1}(N B_k) \}^{2/k} )\} + \psi^{-1}(N B_k) - A_k $$
so that $ NB_k + A_k \geq \max\{ A_k, \exp( \{ \psi^{-1}(N B_k) \}^{2/k} )\}$ since $\psi^{-1}(N B_k)\geq  0 $.  This implies that
$\psi^{-1}(N B_k)\leq \log^{k/2} (NB_k+A_k)$. Thus the result holds with $C_k=1+B_k+A_k$.
\end{proof}

\begin{lemma}\label{lem:auxmoment} Suppose that $X$ is a random variable such that $\Ep[X^2]=1$ and $\{\Ep[|X|^m]\}^{1/m} \leq B_m\gamma$ for $m\geq 2$. Then, for any $k\geq 3$ and $\delta \in (0,2]$ we have $\Ep[|X|^k] \leq \gamma^{k-2+\delta}B_{2(k-2+\delta)/\delta}^{k-2+\delta}$.
\end{lemma}
\begin{proof}
By applying Holder's inequality
$$ \Ep[|X|^k] = \Ep[|X|^{2-\delta}|X|^{k-2+\delta}]\leq \{\Ep[|X|^2]\}^{(2-\delta)/2} \ \ \{\Ep[ |X|^{2(k-2+\delta)/\delta}]\}^{\delta/2}$$
where $2(k-2+\delta)/\delta\geq 2$. The result follows since by assumption $\Ep[|X|^2]=1$ and
$$\{\Ep[ |X|^{2(k-2+\delta)/\delta}]\}^{\delta/2} \leq \{B_{2(k-2+\delta)/\delta}^{2(k-2+\delta)/\delta}\gamma^{2(k-2+\delta)/\delta}\}^{\delta/2} = B_{2(k-2+\delta)/\delta}^{k-2+\delta}\gamma^{k-2+\delta}.$$\end{proof}

\begin{lemma}\label{aux:barpsi}
Let $m_j := \sigma_j^{-1}\Sigma_j^{-1}(1+\|\beta_0\|)(1+\|\mu^j_0\|)$. Under Condition A, for any $t>0$ we have
$$ P( \max_{j\in S} |\bar\psi_j(y_1,z_1)| > t ) \leq 2|S|\exp( - tc_{w,\xi,x}/\max_{j\in S}m_j)$$
 where $c_{w,\xi,x}$ depends only on the subgaussian parameters of $w$, $\xi$, and $x$. Moreover, for any $k\geq 1$ we have $$ \Ep[|\bar \psi_j(y_1,z_1)|^k]^{1/k} \leq k C_{w,\xi,x} \sigma_j^{-1}\Sigma_j^{-1}\{1+\|\mu^j_0\|\}\{1+\|\beta_0\|\}  \ \ \mbox{and}$$
 $$\Ep[\max_{j\in S}|\bar \psi_j(y_1,z_1)|^k]^{1/k} \leq C_k\log (4|S|) \max_{j\in S}  \sigma_j^{-1}\Sigma_j^{-1}\{1+\|\mu^j_0\|\}\{1+\|\beta_0\|\}.$$ \end{lemma}
\begin{proof}
Note that $\Ep[\bar\psi_j(y_1,z_1)]=0$ and $\Ep[\bar\psi_j^2(y_1,z_1) ]=1$. Moreover, for each $k\geq 1$  we have
$$\begin{array}{rl}
\Ep[|\bar \psi_j(y_1,z_1)|^k]^{1/k} & \leq \sigma_j^{-1}\Sigma_j^{-1}\Ep[|(e^j-\mu^j_0)^Tz_1|^{2k}]^{1/2k}\Ep[|\xi_1-w_1^T\beta_0|^{2k}]^{1/2k} \\
 & + |(e^j-\mu^j_0)^T\Gamma\beta_0| \\
& \leq \sigma_j^{-1}\Sigma_j^{-1}\{1+\|\mu^j_0\|\}\{1+\|\beta_0\|\}  C_{w,\xi,x} k
 \end{array}$$ since $\Ep[|(e^j-\mu^j_0)^Tz_1|^{2k}]^{1/2k}\leq C_z \sqrt{2k}(1+\|\mu^j_0\|)$ and $\Ep[|\xi_1-w_1^T\beta_0|^{2k}]^{1/2k}\leq C_{\xi,w}\sqrt{2k}(1+\|\beta_0\|)$ by the subgaussian assumption and Lemma \ref{lem:subgaussian}. Here $C_{w,\xi,x}$ depends only on the subgaussian parameters of $w$, $\xi$, and $x$.
Therefore, $\|\bar \psi_j(y_1,z_1)\|_{\psi_1}\leq \sigma_j^{-1}\Sigma_j^{-1}\{1+\|\mu^j_0\|\}\{1+\|\beta_0\|\}  C_{w,\xi,x}=m_jC_{w,\xi,x}$. In turn we have that for some universal constant $c>0$
$$ P(|\bar \psi_j(y_1,z_1)|>t)\leq 2\exp(-ct/m_jC_{w,\xi,x})$$
and the result follows from the union bound and setting $c_{w,\xi,x}=c/C_{w,\xi,x}$.

Next note that for $\psi(t) = \max\{ A_k, \exp(t^{1/k})\}+t-A_k$ where $A_k=\exp(k-1)$, we have $\psi(0)=0$, $\psi$ non-negative, convex, strictly increasing and $\psi(t)\leq \exp(t^{1/k})+t$. Moreover let $M$ be such that $\Ep[ \exp(|\bar \psi_j(y_1,z_1)|/M)]\leq 2$ where $M\geq 3 k C_{w,\xi,x}\sigma_j^{-1}\Sigma_j^{-1}(1+\|\mu^j_0)(1+\|\beta_0\|)$. Then
$$\begin{array}{rl}
\psi(\Ep[\max_{j\in S}|\bar \psi_j(y_1,z_1)|^k/M^k]) &\leq \Ep[\max_{j\in S} \psi(|\bar \psi_j(y_1,z_1)|^k/M^k)] \\
& \leq |S|\max_{j\in S} \Ep[\psi(|\bar \psi_j(y_1,z_1)|^k/M^k)]\\
& \leq |S|\max_{j\in S} \Ep[\exp(|\bar \psi_j(y_1,z_1)|/M)]\\
&+ |S|\max_{j\in S}\Ep[|\bar \psi_j(y_1,z_1)|^k/M^k]\\
& \leq 4|S|
\end{array}$$ where the last relation follows by definition of $M$.
Thus $$\Ep[\max_{j\in S}|\bar \psi_j(y_1,z_1)|^k] \leq M^k \psi^{-1}(|S|B_k).$$ To bound the inverse function we have
By definition we have $|S|B_k +A_k \geq \max\{A_k, \exp( |\psi^{-1}(|S|B_k)|^{1/k} )\}$ which implies
$$\psi^{-1}(|S|B_k) \leq \log^k(4|S|+A_k).$$
\end{proof}

Let $ ( W_i)_{i=1}^n$ be a sequence of independent copies of a random element $ W$  taking values in a measurable space $({\mathcal{W}}, \mathcal{A}_{{\mathcal{W}}})$ according to a probability law $P$. Let $\mathcal{F}$ be a set  of suitably measurable functions $f\colon {\mathcal{W}} \to \mathbb{R}$, equipped with a measurable envelope $F\colon \mathcal{W} \to \mathbb{R}$. Let $\Gn(f)=n^{-1/2}\sum_{i=1}^n f(W_i)-\Ep[f(W_i)]$.

  \begin{lemma}[Maximal Inequality adapted from \cite{chernozhukov2012gaussian}]
\label{lemma:CCK}  Suppose that $F\geq \sup_{f \in \mathcal{F}}|f|$ is a measurable envelope for the finite class $\mathcal{F}$
with $\| F\|_{P,q} < \infty$ for some $q \geq 2$.  Let $M = \max_{i\leq n} F(W_i)$ and $\sigma^{2} > 0$ be any positive constant such that $\sup_{f \in |\mathcal{F}|}  \| f \|_{P,2}^{2} \leq \sigma^{2} \leq \| F \|_{P,2}^{2}$. Then
\begin{equation*}
\Ep_P [ \max_{f \in \mathcal{F}} |\Gn(f)| ] \leq K  \left( \sqrt{\sigma^{2} \log \left ( \frac{|\mathcal{F}|\| F \|_{P,2}}{\sigma} \right ) } + \frac{\| M \|_{P, 2}}{\sqrt{n}} \log \left ( |\mathcal{F}|  \frac{\| F \|_{P,2}}{\sigma} \right ) \right),
\end{equation*}
where $K$ is an absolute constant.  Moreover, for every $t \geq 1$, with probability $> 1-t^{-q/2}$,
\begin{multline*}
\max_{f \in \mathcal{F}} |\Gn(f)|  \leq (1+\alpha) \Ep_P [ \max_{f \in \mathcal{F}} |\Gn(f)|  ] \\
+ K(q) \Big [ (\sigma + n^{-1/2} \| M \|_{P,q}) \sqrt{t}
+  \alpha^{-1}  n^{-1/2} \| M \|_{P,2}t \Big ], \ \forall \alpha > 0,
\end{multline*}
where $K(q) > 0$ is a constant depending only on $q$.
\end{lemma}

\end{appendix}

\bibliography{EIVbib}
\bibliographystyle{plain}

\end{document}